\documentclass{amsart}

\usepackage{style}

\title[K3 surfaces with two involutions and low Picard number]{K3 surfaces with two involutions and low Picard number}

\author{Dino Festi, Wim Nijgh, Daniel Platt}

\address{Dipartimento di matematica ``Tullio Levi-Civita'', Universit\`a degli studi di Padova, via Trieste 63, 35121, Padova, Italy.}
\email{dino.festi@unipd.it}

\address{Mathematical Institute, Leiden University, Postbus 9512, 2300 RA Leiden, The Netherlands.}
\email{w.nijgh@math.leidenuniv.nl}

\address{Department of Mathematics, King's College London, Strand, London WC2R 2LS, United Kingdom.}
\email{daniel.platt.berlin@gmail.com}

\date{\today}

\begin{document}

\begin{abstract}
    Let $X$ be a complex algebraic K3 surface of degree $2d$ and with Picard number $\rho$.
    Assume that $X$ admits two commuting involutions: one holomorphic and one anti-holomorphic.
    In that case, $\rho \geq 1$ when $d=1$ and $\rho \geq 2$ when $d \geq 2$.
    For $d=1$, the first example defined over $\IQ$ with $\rho=1$ was produced already in 2008 by Elsenhans and Jahnel.
    A K3 surface provided by Kond\={o}, also defined over $\IQ$, can be used to realise the minimum $\rho=2$ for all $d\geq 2$.
    In these notes we construct new explicit examples of K3 surfaces over the rational numbers realising the minimum $\rho=2$ for $d=2,3,4$.
    We also show that a nodal quartic surface can be used to realise the minimum $\rho=2$ for infinitely many different values of $d$.
    Finally, we strengthen a result of Morrison by showing that for any even lattice $N$ of rank $1\leq r \leq 10$ and signature $(1,r-1)$ there exists a K3 surface $Y$ defined over $\IR$ such that $\Pic Y_\IC=\Pic Y \cong N$.
\end{abstract}

\maketitle

\section{Introduction}

One of the central objects of study in algebraic geometry are K3 surfaces, which 
benefit of much interest also from other areas of mathematics.
In differential geometry, one landmark result is Yau's proof of the Calabi conjecture in \cite{Yau1977,Yau1978} and, because of it, complex K3 surfaces are known to admit metrics with holonomy equal to $SU(2)$. Through this connection with holonomy groups, K3~surfaces often appear in the study of higher-dimensional manifolds with special holonomy.

One motivation for this article is the differential geometry of $G_2$-manifolds, i.e. the geometry of manifolds of real dimension $7$ with holonomy equal to $G_2$. 
Here the group $G_2$ denotes the automorphism group of the octonions, see \cite{SW2017}.
Only a few ways of constructing $G_2$-manifolds are currently known, and one such way is the blowup procedure by Joyce and Karigiannis \cite[Section 7.3]{Joyce2017}, 
which uses a complex K3 surface $X$ with two commuting involutions, 
i.e. bijective maps which are their own inverse.
More precisely, one needs a non-symplectic holomorphic involution $\iota: X \rightarrow X$, and an anti-holomorphic involution $\sigma: X \rightarrow X$. 
Another manifold construction by Kovalev and Lee \cite{KL2011} uses a K3 surface with only one (non-symplectic holomorphic) involution.
By using different models of K3 surfaces, such as quartics or double planes ramified over a sextic, one obtains many examples of $G_2$-manifolds.

Not much is known about which manifolds admit metrics with holonomy $G_2$, or how many such metrics there are on a manifold that does.
Because of this, much work in the field is done on constructing and studying examples.
Using explicit examples of K3 surfaces, one obtains good information about the metric on the resulting $G_2$-manifolds.
In the future, this may be helpful for studying possible metric degenerations of $G_2$-manifolds, similar to \cite[Example 7.2]{Joyce2017}), which is realised by two different construction methods, exhibiting two different metric degenerations.
The analogous question for K3 surfaces had likewise first been studied via examples, but by now some general results exist, see the references in Chen, Viaclovsky, and Zhang \cite[Section 1]{Che2020} for an overview.

On the other hand, a program proposed by Donaldson and Thomas~\cite{Donaldson1998} to study the moduli space of $G_2$-metrics suggests counting $G_2$-instantons, i.e. solutions to a certain partial differential equation on principal bundles over the given manifold.
If the manifold with holonomy $G_2$ is obtained via one of the two constructions above, one gets $G_2$-instantons from certain stable bundles over the K3 surface, see Walpuski-Sá Earp \cite{WS2015}, Walpuski \cite{Wal2016}, and the last author \cite[Section~5.2]{Platt2022}.
Even more than for $G_2$-manifolds, there are few general results about $G_2$-instantons and current work concentrates on constructing and studying examples.
The instanton constructions cited above use detailed information about the K3 surfaces in question, so an ample supply of explicit examples of K3 surfaces is helpful for these constructions.
The main challenge in this research area is the compactification of the moduli space of $G_2$-instantons.
As for metrics, it is hoped that explicit examples of degenerations will suggest candidates for a compactification.

As a general principle, checking whether a given bundle is stable becomes computationally harder the more line bundles there are on the K3 surface.
This can be seen in practice in the work of Jardim, Menet, Prata and S\'a Earp \cite[Theorem 3]{Jardim2017}.

The group of line bundles modulo isomorphism on a complex K3 surface $X$ is called the \emph{Picard group}, and we will denote it by $\Pic X$.
It is a finitely generated free abelian group, and its rank, which we denote by $\rho(X)$, is called the \emph{Picard number} of $X$.
The Picard group can be endowed with a non-degenerate bilinear form, making it a lattice (that is, a finitely generated free abelian group with a non-degenerate bilinear form) called the \emph{Picard lattice}. 
A lattice can be identified by a matrix, called its Gram matrix, see~\autoref{r:Gram} for definition and notations.
We are then interested in examples of K3 surfaces with low Picard number and two commuting involutions, one holomorphic and one non-holomorphic.
What is the lowest Picard number possible for such K3 surfaces, for any given degree?
Combining results of Kondo and Elsenhans--Jahnel,  \cite{Kondo1992,EJ08}, we give the following complete answer to this question.

\begin{theorem}
    \label{c:rho-bounds}
    Let $X$ be a K3 surface of degree $2d$ and with Picard rank $\rho$, admitting a holomorphic and an anti-holomorphic involution which commute.
    If $d=1$, then $\rho \geq 1$ and there exist examples defined over $\IQ$ with $\rho=1$.
    If $d>1$, then $\rho \geq 2$ and there exist examples over~$\IQ$ with $\rho=2$ and Picard lattice $[0 \; 1 \; 0]$.
\end{theorem}

The proof of the theorem uses an explicit example of a K3 surface which can be endowed with multiple polarisations.
As mentioned above, it is desirable to have many more explicit examples for applications in $G_2$-geometry.
We therefore give the following explicit examples.

\begin{example}
    \label{e:our-new-K3-examples}
    We present the following new examples:
    
    \begin{enumerate}    
        \item 
        in \autoref{s:an-example},
        a construction method 
        for K3 surfaces over $\IQ$ of degree $4$ and $8$ with Picard lattice $[4 \; 5 \; 2]$,
        together with an explicit example;
    
        \item 
        in \autoref{s:another-example},
        a construction method  
        for K3 surfaces over $\IQ$ of degree $6$ with Picard lattice $[6 \; 6 \; 2]$,
        together with an explicit example;
    
        \item 
        in \autoref{s:RealK3},
        for every $d>3$, 
        we show the existence of K3 surfaces over $\IR$ with Picard lattice  $[2 \; d{+}1 \; 2d]$ and polarization of degree $2d$.
    \end{enumerate}
    We also review the following known K3 surface: 
    \begin{enumerate} 
        \item[(4)] 
        in \autoref{s:NodalQuartics}, a K3 surface defined over $\IQ$  with Picard lattice $[4 \; 0 \; {-}2]$ admitting infinitely many polarizations of different degrees. 
    \end{enumerate}
\end{example}

We say that a complex K3 surface $X$ \emph{can be defined over a field} $k$ if $X$ admits a projective model, whose defining equations have coefficients contained in $k$. 
The existence of an anti-holomorphic involution on $X$ is guaranteed as soon as the surface can be defined over~$\IR$: 
indeed, in this case, the complex conjugation provides an anti-holomorphic involution. 
In the following, we will always use this one as it will commute with the holomorphic involutions we will provide.
Conversely, if there exists an anti-holomorphic involution on $X$, then $X$ can be defined over $\IR$ (see Silhol's \cite[Proposition 1.3]{Sil89}).

A biholomorphic map on a complex K3 surface $X$ comes from an automorphism on the associated algebraic surface. 
A holomorphic involution on $X$ will commute with the anti-holomorphic one if and only if the associated automorphism on the model of $X$ over $\IR$ can be defined over $\IR$. 
As a conclusion, we have the following important remark.

\begin{remark}\label{r:anti-holo}
    A complex K3 surface has a commuting holomorphic and anti-holomorphic involution if and only if the underlying algebraic K3 surface can be defined over $\IR$ and admits an automorphism of order 2.
    This observation gives the main idea of this work: 
    we look for K3 surfaces that can be defined over $\IR$ admitting an ample divisor $D$ with $D^2=2$.
    This ample divisor will provide the automorphism of order $2$ we are after, see~\autoref{l:Deg2K3}.
\end{remark}

The paper is structured as follows.
In~\autoref{s:Background} we introduce K3 surfaces and their Picard lattice, reviewing basic properties and results.
In the same section, we summarise the known results about involutions of K3 surfaces.
In~\autoref{s:DoubleCovers} we review branched double covers of $\IP^2$ which play a central role in the rest of the article and  we prove \autoref{c:rho-bounds}. 
In~\autoref{s:NodalQuartics} we study nodal quartics, serving as a link between K3 surfaces of degree $2$ and higher degrees as they provide a construction of a K3 surface over $\IQ$ with Picard rank $2$ realising infinitely many degrees.
Smooth quartics admitting a holomorphic involution are studied in~\autoref{s:smooth-quartics-with-picard-number-2}, where we prove that there are infinitely many non-isomorphic smooth quartics over $\IC$ with Picard number $2$ and admitting a holomorphic involution.
We also show that if a quartic surface has a holomorphic involution that can be extended to a linear involution of its ambient space $\IP^3$, then its Picard number is higher.
We construct examples of K3 surfaces of degrees~$2d$, with Picard number $2$, defined over $\IQ$, and admitting a holomorphic involution,
for $2d=4,8$ in~\autoref{s:an-example} and for $2d=6$ in~\autoref{s:another-example}.
In~\autoref{s:RealK3} we study when complex K3 surfaces can be defined over the real numbers and prove that any even lattice $N$ of rank $1\leq r \leq 10$ and signature $(1,r-1)$ can be realized as Picard lattice of a K3 surface defined over $\IR$. 
We use this result to show the existence of infinitely many non-isomorphic K3 surfaces defined over $\IR$ with $\rho=2$ admitting an involution.
The \texttt{Magma} code pertaining to our explicit examples can be found online at \href{http://github.com/danielplatt/quartic-k3-with-involution}{github.com/danielplatt/quartic-k3-with-involution}.

\section*{Acknowledgments}
We would like to thank Alex Degtyarev, Alice Garbagnati, Bert van Geemen, Ronald van Luijk, Bartosz Naskr\c{e}cki, and Simon Salamon for many useful conversations on this topic.
We also thank the anonymous referee for their comments.
The first author was partially supported by the PRIN grant \texttt{PRIN202022AGARB\char`_01} while in Milano.

\section{Some background}\label{s:Background}

In this section, we provide the necessary background about K3 surfaces. We mostly follow Huybrecht~\cite{Huy16} and Kondo~\cite{Kon20}; 
for basic notions and results in algebraic geometry we refer to Hartshorne~\cite{Har77} and Griffith~\cite{Gri94}.
As a start, we give a formal definition of a complex K3 surface.

By complex K3 surface we mean a compact connected complex manifold~$X$ of dimension two with trivial canonical bundle and $H^1(X,\mathcal{O}_X)=0$. It follows that $X$ admits a holomorphic $2$-form that is nowhere vanishing and unique up to scaling. 
Examples of K3 surfaces are, among others, 
smooth quartic surfaces in~$\IP^3$, 
smooth intersections of a cubic and a quadric in $\IP^4$,
smooth intersections of three quadrics in~$\IP^5$,
and double covers of~$\IP^2$ ramified over a smooth curve of degree six.

An algebraic K3 surface $X$ over a field~$k$ is a projective, geometrically integral and smooth variety~$X$ over~$k$ such that its canonical bundle $\Omega_{X/k}^2$ is isomorphic to $\mathcal{O}_X$ and with $H^1(X,\mathcal{O}_X)=0$. 
By the GAGA theorem \cite{Ser1956}, we have the following connection between algebraic and complex K3 surfaces:
starting with an algebraic K3 surface $X$ over~$k$, where $k$ is a subfield of $\IC$, the complex points~$X(\IC)$ have a natural structure of a complex K3 surface; as a converse, if one starts with a \emph{projective} complex K3 surface $X$, then there exists an algebraic K3 surface $X'$ over $\IC$ such that $X'(\IC)=X$. More specifically, a complex K3 surface comes from an algebraic K3 surface if and only if it can be embedded in projective space.

\begin{remark}\label{r:NonProj}
  There are complex K3 surfaces that are not projective and hence not algebraic (in fact non-projective K3 surfaces are dense in the moduli space of complex K3 surfaces).
  Nikulin shows that if a K3 surface has a non-symplectic automorphism of finite order, then it is projective, see~\cite[Theorems 0.1a) and 3.1a)]{Nikulin1979}.
As in both Joyce--Karigiannis' and Kovalev--Lee's constructions a \emph{non-symplectic} involution is needed, non-projective K3 surface will not be considered in this paper.
\end{remark}

Because of~\autoref{r:NonProj}, our attention will focus on algebraic complex K3 surfaces. 
Therefore, unless stated otherwise, 
a complex K3 surface will be assumed to be algebraic. We will use both the algebraic and manifold structure on the K3 surface.

\subsection{The Picard lattice of a K3 surface}
\label{ss:PicLattice}

Let $k$ be a field and let~$\overline{k}$ be an algebraic closure of~$k$. Let~$X$ denote an algebraic K3 surface over~$k$. 
Throughout this paper, we will denote by~$X_{\overline{k}}$ the base change of~$X$ to~$\overline{k}$. 
If~$k$ is a subfield of~$\IC$, e.g. $k=\IQ$ or $k=\IR$, we use the notation~$X_\IC$ to denote the base change of~$X$ to~$\IC$.

The \emph{Picard group} of~$X$, denoted~$\Pic X$, is the group of isomorphism classes of invertible sheaves on~$X$ under the tensor product.
It is known that the Picard group is isomorphic to~$H^1(X,\mathcal{O}_X^*)$, where~$\mathcal{O}_X^*$ denotes the sheaf whose sections over an open set $U$ are the units in the ring~$\mathcal{O}_X(U)$. 
There is also a natural isomorphism $\Cl X\cong \Pic X$, where~$\Cl X:=\Div X/{\sim}$ denotes the divisor class group modulo linear equivalence \cite[Corollary~II.6.16]{Har77}.

There is a unique symmetric bilinear pairing $\Div X_{\overline{k}} \times \Div X_{\overline{k}}\to \IZ$, sending any two divisors~$C,D$ on~$X_{\overline{k}}$ to the integer $C.D$, such that if~$C$ and~$D$ are non-singular curves meeting transversally, the number~$C.D$ is the number of points of~$C\cap D$;
for a divisor $D$ the number $D^2:=D.D$ is called the \emph{self-intersection} of $D$.
The pairing depends only on the linear equivalence classes and so it descends to a pairing $\Pic X_{\overline{k}} \times \Pic X_{\overline{k}} \to \IZ$ \cite[Theorem V.1.1]{Har77}. 
We can identify $\Pic X$ as a subgroup of $\Pic X_{\overline{k}}$, and this pairing on $\Pic X_{\overline{k}}$ will then induce a pairing on $\Pic X$.

If $k$ is a subfield of $\IC$, then using the GAGA theorem (see \cite{Ser1956}), there is a natural identification of the group $\Pic X_\IC$ with the group of line bundles on a topological space and so this definition is equivalent with the definition given in the introduction.
If $k$ is a subfield of $\IC$, we define the \textit{Picard number} of $X$ to be the Picard number of the associated complex surface as defined in the introduction, or equivalently, to be the rank of $\Pic X_\IC$.

\begin{remark}
    Although we can identify $\Pic X$ as a subgroup of $\Pic X_{\overline{k}}$, it is not always the case that they have the same rank. 
    In the explicit examples we are constructing, we will prove that $\Pic X=\Pic X_{\IC}$, but in general this is definitely not the case. 
    To avoid any confusion, we will use the definition Picard number only for complex K3 surfaces. 
\end{remark}

A \textit{lattice}~$\Lambda$ is a finitely generated free abelian group together with a non-degenerate bilinear pairing $\Lambda\times \Lambda\to \IZ$.
The Picard group of a projective K3 surface is finitely generated and torsion free~\cite[\S 1.2]{Huy16}, and the induced intersection pairing is non-degenerate~\cite[Proposition 1.2.4]{Huy16}, hence making the group $\Pic X$ a lattice. We refer to this as the \emph{Picard lattice} of~$X$.  
If~$D$ is the class of a curve of arithmetic genus $p_a$, the adjunction formula shows that $D^2=2p_a-2$. It follows that~$\Pic X$ is an \emph{even} lattice, i.e., $D^2$ is even for all~$D$ in~$\Pic X$. 

\begin{remark}\label{r:Gram}
    Choosing a basis for $\Pic X$, we can associate a \emph{Gram matrix} to the bilinear pairing. We introduce the notation~$[a\; b\; c]$ with~$a,b,c\in \IZ$ for a lattice of rank 2 and a basis with Gram matrix equal to $\begin{pmatrix}
        a & b\\
        b & c
    \end{pmatrix}$.
    We also introduce the notation $\langle a \rangle$ for a lattice of rank 1 and a basis with Gram matrix equal to $\begin{pmatrix}
        a
    \end{pmatrix}$.
\end{remark}

\subsection{The second cohomology group and the Hodge decomposition}

Let $X$ denote a complex K3 surface. Using the exponential sequence 
$$
0\to\IZ\xrightarrow{2\pi i}\cO_{X(\IC)}\xrightarrow{\text{exp}}\cO_{X(\IC)}^*\to 0
$$ 
we get an induced injective map from the Picard group of a complex K3 surface~$X$ to~$H^2(X,\IZ)$, denoted~$c_1\colon \Pic X\hookrightarrow H^2(X,\IZ)$. 
The cup product gives~$H^2(X,\IZ)$ the structure of a lattice, of which the restriction to the Picard group is exactly the intersection pairing. 

It is known that the lattice $H^2(X,\IZ)$ is isomorphic to the \emph{K3 lattice} 
$$
\Lambda_{\text{K}3}:=U^{\oplus 3} \oplus E_8(-1)^{\oplus 2},
$$
where $U$ is the hyperbolic  rank $2$ lattice $[0 \; 1 \; 0]$, and $E_8(-1)$ is the unique even unimodular negative-definite lattice of rank $8$ \cite[Proposition 1.35]{Huy16}. This is a lattice of signature $(3,19)$. By the Hodge index theorem \cite[Subsection 1.2.2]{Huy16}, the signature of the Picard lattice is~$(1,\rho(X)-1)$. It follows that $\rho(X)\leq 20$.

The isomorphism between the spaces $H^2(X,\IZ)$ and $\Lambda_{\text{K}3}$ is not canonical. Choosing an isometry $\alpha_X\colon H^2(X,\IZ)\to \Lambda_{\text{K}3}$ is called a \emph{marking} of $X$, and we call the tuple $(X,\alpha_X)$ a~\emph{marked}~K3 surface. For the remainder of this subsection, 
we will always assume that $X$ is marked by a marking $\alpha_X$.

We have a natural identification $H^2(X,\IZ)\otimes \IC\cong H^2(X,\IC)$. 
As a complex K3 surface is a K\"ahler manifold, there is a Hodge decomposition 
$$
H^2(X,\IC)=\bigoplus_{p+q=2}H^{p+q}(X)\, ,
$$ 
where $H^{1,1}(X)$ can be defined over $\IR$ and $H^{2,0}(X)$ and $H^{0,2}(X)$ are complex conjugates of each other. Moreover, there are natural isomorphisms $H^{p,q}(X)\cong H^q(X,\Omega_X^p)$. 
By the Lefschetz theorem on~$(1,1)$-classes \cite[p.163]{Gri94}, we have that the image of the map $c_1$ equals the set $H^{1,1}(X)\cap H^2(X,\IZ)$. 

Next we fix a nowhere vanishing holomorphic $2$-form $\omega_X\in H^0(X,\Omega_X^2)\subset H^2(X,\IC)$. 
As $H^0(X,\Omega_X^2)$ is $1$-dimensional by definition of a K3 surface, we have 
$$
H^0(X,\Omega_X^2)=\langle \omega_X \rangle\cong \IC\; .
$$ 
Extending the cup product to $H^2(X,\IC)$, the form $\omega_X$ satisfies the \textit{Riemann conditions}  
\begin{equation}\label{eq:Rc}
    \omega_X^2=0 \ \text{ and } \ \omega_X\cdot \overline{\omega}_X>0\, ,
\end{equation} 
see \cite[p.57]{Kon20}. Observe that $H^{2,0}\oplus H^{0,2}(X)=\langle \text{Re}(\omega_X),\text{Im}(\omega_X)\rangle$ and $H^{1,1}(X)$ will be orthogonal to this subspace. In particular, $\omega_X$ determines the Hodge decomposition of $H^2(X,\IC)$.

We define the \textit{transcendental lattice} of $X$, denoted $T_X$, to be the lattice~$(\Pic X)^\perp\subseteq H^2(X,\IZ)$. 
By the Lefschetz theorem on~$(1,1)$-classes and the construction above, one can deduce that the~$2$-form~$\omega_X$ is contained in $T_X\otimes \IC$.

\subsection{Reduction maps and upper bounds for the Picard number}

Let~$X$ be a K3 surface defined over~$\IQ$.
Then~$X$ can be defined by equations with coefficients in~$\IZ$.
Let~$p$ be a prime and let~$X_p$ denote the variety obtained considering the equations defining~$X$ over~$\IZ$ modulo~$p$.
We call~$X_p$ the \emph{reduction} of~$X$ modulo $p$.
We say that~$p$ is \emph{a prime of good reduction} for~$X$ if~$X_p$ is smooth, i.e.,~$X_p$ is a K3 surface over $\IF_p$.
Then it is known that there is an embedding $\Pic X_\IC \to \Pic X_{p,\overline{\IF}_p}$.
So the rank of~$\Pic X_{p,\overline{\IF}_p}$ gives an upper bound for the Picard number of~$X$. 
For details we refer to \cite[Subsection~17.2.5]{Huy16}. 

One can calculate the rank of $\Pic X_{p,\overline{\IF}_p}$ explicitly. Let $\ell\neq p$ be another prime. Using étale cohomology we can embed $\Pic X_{p,\overline{\IF}_p}$ into $H^2_\text{ét}(X_{p,\overline{\IF}_p},\IQ_\ell(1))$. 
Let $F_p\in \text{Gal}(\overline{\IF}_p/\IF_p)$ be the Frobenius automorphism. Then this induces an automorphism~$F_p^*$ on~$H^2_\text{ét}(X_{p,\overline{\IF}_p},\IQ_l(1))$. 
The Tate conjecture, which is proven for $K3$ surfaces of positive characteristic by Kim and Pera~\cite{Kim15},
implies that the rank of~$\Pic X_{p,\overline{\IF}_p}$ is exactly the number of eigenvalues~$\lambda$ of~$F_p^*$ such that~$\lambda$ is a root of unity. 
The characteristic polynomial of~$F_p^*$ can be computed explicitly. A practical way to do it by counting points is given by van Luijk~\cite[Section 2]{vL07}. 
There are numerous refinements and developments, see for example the works of Elsenhans--Jahnel and Charles~\cite{EJ11, Cha14}.

\texttt{Magma} has a built-in function to calculate the rank of $\Pic X_{p,\overline{\IF}_p}$ with $p$ of good reduction in the case that $X$ is a double cover of the plane. The function is based on Elsenhans and Jahnel's work~\cite{EJ10}. We will use this to compute upper bounds for our Picard number.

\subsection{The ample cone and polarized K3 surfaces}\label{ss:AmpleCone}

In this subsection, we introduce some basic facts about the ample cone of a~K3~surface;
we follow~\cite[Chapter 8]{Huy16}. Let $X$ be a K3 surface and let~$H$ be an ample divisor on $X$. We define the \emph{positive cone} $\cP^+\subset \Pic X\otimes \IR$ of $X$ as the connected component of the set $\{ \lambda\in \Pic X\otimes\IR \; :\; \lambda^2>0\}$ containing $H$.
The \emph{ample cone} $\cA$ of $X$ is the set of all (non-zero) finite sums $\sum_i a_iL_i$ inside $\Pic X\otimes \IR$, with~$L_i\in\Pic X$ an ample class and $a_i\in\IR_{>0}$.

Assume that $H$ is \emph{primitive}, meaning that the quotient module $\Pic X/\IZ[H]$ has no torsion. Then we call the pair $(X,H)$ a \emph{polarized K3 surface} of degree~$H^2$.
Note that the degree~$H^2$ is even, as~$\Pic X$ is an even lattice.
For example, by simply taking~$H$ to be a hyperplane section, one can see that:  
double covers of~$\IP^2$ ramified above smooth sextic curves are polarized K3 surfaces of degree~2;
smooth quartics in~$\IP^3$ are K3 surfaces of degree~4;
smooth intersections of a cubic and a quadric in~$\IP^4$ have degree~6;
smooth intersections of three quadrics in~$\IP^5$ have degree~8.

Let $(X,H)$ be a polarized K3 surface.
The Nakai--Moishezon--Kleiman criterion says that a class~$D\in\Pic X$ is ample if and only if $D^2>0$ and $D.[C]>0$ for all curves $C\subset X$ \cite[Theorem V.1.10]{Har77}.
For polarized K3 surfaces, this result can be refined (see \cite[Corollary 8.1.6]{Huy16}), yielding that $D$ is ample if and only if
\begin{enumerate}
    \item $D^2>0$,
    \item $D.[C]>0$ for all the smooth \emph{rational} curves $C\subset X$,
    \item $D.H>0$.
\end{enumerate}
It follows that 
$$
\cA=\{ \lambda\in \cP^+\; :\; \lambda . [C]>0 \textrm{ for all smooth rational curves } C\subset X \} \subseteq \cP^+.
$$
This motivates the construction below, called the \emph{chamber decomposition} of $\Pic X$.

\begin{remark}\label{r:no-2} 
  Recall that on a K3 surface, if $C$ is the class of a curve of arithmetic genus $g$, then $C^2=2g-2$.
  This means that rational curves have self-intersection $-2$.

  The reverse is not true. In general a $-2$-class does not need effective. 
  An effective $-2$-class is the class of a rational curve if and only if the class is irreducible.

  If $\Pic X$ does not represent $-2$, i.e. there are no $-2$-classes,  then there are no rational curves either. 
  It follows that $\cA = \cP^+$.
\end{remark}

Let $\Delta:=\{ \delta \in \Pic X \; : \; \delta^2=-2\}$ be the set of \emph{roots} of $\Pic X$.
If $\delta\in \Delta$, then using the Riemann--Roch theorem one can see that either $\delta$ or $-\delta$ is effective.
We define $\Delta_+:=\{\delta\in\Delta \; : \;  \delta \text{ is effective}\}$.
Clearly the equality $\Delta=\Delta_+ \sqcup \Delta_-$ holds,
where $\Delta_-:=-\Delta_+$.
For any root~$\delta\in\Delta$ we define the \emph{reflection} 
$$
s_\delta\colon \Pic X\to \Pic X, \;\; x\mapsto x+(x . \delta)\delta\, .
$$
Notice that $s_\delta$ is an involution. By linearity it can be extended to $\Pic  X\otimes \IR$ and it preserves the intersection pairing, and so it preserves~$\cP^+$.
We call $\Fix (s_\delta)=\cP^+\cap \delta^\perp$ the \emph{wall} associated with~$\delta$.
The connected components of $\cP^+ \setminus \bigcup_{\delta\in\Delta} \delta^\perp$ are called the \emph{chambers} of $\cP^+$.
We define the \emph{Weyl group} $W$ as the group of isometries generated by the reflections,
$$
W:=\langle s_\delta \; :\; \delta\in \Delta_+\rangle\, .
$$
The group $W$ acts transitively on the set of chambers of $\cP^+$  \cite[Proposition 8.2.6]{Huy16}.
From the refined Nakai--Moishezon--Kleiman criterion above, it follows that the ample cone is the chamber whose walls are given by $[C]^\perp$, for every smooth rational curve $C\subset X$. 
Moreover, none of the conditions that we get from the $-2$-classes of these curves are superfluous, see also~\cite[Remark 8.2.8]{Huy16}.

\begin{remark}\label{r:-2classcurve}
  If $\rho(X)=2$, the above implies that there are at most two smooth rational curves on $X$ up to linear equivalence. 
\end{remark}

\subsection{Moduli spaces of K3 surfaces}\label{ss:moduli}

Next we will set up the theory of moduli spaces of K3 surfaces over $\IC$. We will closely follow \cite[Chapter 6]{Kon20} and refer to this for more background.

The \emph{period domain of complex K3 surfaces} is the set $$\Omega:=\{[\omega]\in \IP(\Lambda_{\text{K}3}\otimes\IC): \langle\omega,\omega\rangle=0, \langle \omega,\Bar{\omega}\rangle >0\}. $$
Let $(X,\alpha_X)$ be a marked K3 surface over $\IC$. 
By the Riemann condition, we have $[\alpha(\omega_X)]\in \Omega$ and we call this point the \emph{period} of  $(X,\alpha_X)$. In particular, if $\cM$ is the set of isomorphism classes of marked K3 surfaces, this gives us a map $\lambda\colon \cM\to \Omega$, called the \emph{period map} of marked K3 surfaces. It is a deep theorem on complex K3 surfaces that this map is surjective, see \cite[Theorem~6.9 and proof in chapter 7]{Kon20}. 
The surjectivity of the map $\lambda$ is only guaranteed if we also take non-algebraic complex K3 surfaces into account. In fact, if we take a general element $[\omega]\in \Omega$, a marked K3 surface~$(X,\alpha_X)$ with period $[\omega]$ will be non-algebraic with $\Pic X=\{0\}$.

Next we will look at the period map for polarized K3 surfaces and their moduli spaces.
Let~$d$ be a positive integer and take a primitive element $h\in \Lambda_{\text{K}3}$ with $h^2=2d$. Denote by~$\Lambda_{2d}$ the orthogonal complement of~$h$ in~$\Lambda_{\text{K}3}$. Then $\Lambda_{2d}$ has signature $(2,19)$ and the isomorphism class of~$\Lambda_{2d}$ is independent of the choice of $h$. 

We define the \emph{period domain of polarized K3 surfaces of degree $2d$} as the set $$\Omega_{2d}:=\{[\omega]\in \IP(\Lambda_{2d}\otimes\IC): \langle\omega,\omega\rangle=0, \langle \omega,\Bar{\omega}\rangle >0\}\subset \Omega. $$
The group $\Gamma_{2d}:=\{\gamma\colon 
\Lambda_{\text{K}3} \xrightarrow{\sim} \Lambda_{\text{K}3} \colon \gamma(h)=h\}$ 
of automorphisms of $\Lambda_{\text{K}3}$ that fix $h$ acts on $\Omega_{2d}$ and one can show that the quotient $\Omega_{2d}/\Gamma_{2d}$ is a variety. This quotient can be identified as the \emph{moduli space of polarized K3 surfaces of degree 2d}. We will motivate this next.

Let $\mathcal{M}_{2d}$ denote the set of isomorphism classes of polarized K3 surfaces of degree~$2d$. 
Let~$(X,H)$ be a polarized K3 surface in $\mathcal{M}_{2d}$ and let $h=\alpha_X(H)$ denote the image of the ample divisor $H$ under a marking $\alpha_X$ of $X$. 
As~$\omega_X$ lies in $T_X\otimes \IC$, it is perpendicular to $\Pic X$ and therefore to $H$ in particular.
This means that~$[\alpha_X(\omega_X)]\in \Omega_{2d}$. 
This association defines a (set-theoretical) map $\lambda_{2d}\colon \mathcal{M}_{2d}\to \Omega_{2d}/\Gamma_{2d}$ called the \emph{period map of K3 surfaces of degree $2d$}. 
This map is injective by the Torelli theorem for polarized K3 surfaces, see   \cite[Theorem 6.5]{Kon20} for a modern formulation and proof.
With respect to a generalized definition of polarized K3 surface, this map is also surjective. We refer to \cite[Theorem~6.12 \& Section 7.3]{Kon20} for details and a sketch of the proof.

\subsection{Involutions on K3 surfaces}

Automorphisms of K3 surfaces have been widely studied, see e.g. \cite[\S 15]{Huy16} for an overview.
Given a K3 surface~$X$, and any automorphism $g\colon X\to X$, there is some~$a\in \IC^*$ such that $g^*(\omega_X)=a\cdot \omega_X$ for any non-zero holomorphic 2-form~$\omega_X$ on~$X$. If~$a=1$, we call an automorphism \emph{symplectic}. Otherwise, we will call it \emph{non-symplectic}.

Let $g\colon X\to X$ be an automorphism of a K3 surface that is an involution. Observe that $g$ is non-symplectic if and only if $g^*(\omega_X)=-\omega_X$. Nikulin showed that if $g$ is symplectic, then any fixed point is isolated \cite[\S 5]{Nikulin1979}, and if $g$ is non-symplectic, then the fixed point set is empty or consists of a disjoint union of non-singular curves  \cite[Theorem 4.2.2]{Nik83}. From the proof of these statements, we can deduce a lower bound for the Picard number. This is summarized in the following theorem.

\begin{theorem}\label{t:NikulinInv}
Let $X$ be a complex K3 surface with a holomorphic involution $g$;
let $F$ be the set of fixed points of $g$.
\begin{itemize}
    \item If $g$ is symplectic, then $F$ consists of eight isolated points and $\rho (X)\geq 9$.
    \item If $g$ is non-symplectic,
    then one of the following holds.
    \begin{enumerate}
      \item $F=\emptyset$. In this case $\rho (X)\geq 10$.
      \item $F$ is a disjoint union of two non-singular elliptic curves. 
      In this case $\rho(X)\geq 10$.
      \item $F=C\sqcup E_1 \sqcup \cdots \sqcup E_k$, where $C$ is a non-singular curve and $E_1,\ldots,E_k$ are non-singular rational curves.
      In this case $\rho(X) \geq 11-p_a(C)+k$, where $p_a(C)$ denotes the arithmetic genus of~$C$.
    \end{enumerate}
\end{itemize}
\end{theorem}
\begin{proof}
If $g$ is symplectic, then Nikulin~\cite[\S 5]{Nikulin1979} shows that $F$ consists of eight isolated points;
in this case, $\rho (X)\geq 9$, as shown in~\cite[\S 15.1]{Huy16}.

In \cite[Theorem 4.2.2]{Nik83},  the fixed point sets of non-symplectic involutions are classified as stated in (1)---(3).
To obtain the bound of the Picard number,
consider the lattice 
$$
L_+:=\{ x\in H^2(X,\IZ) \, :\, g^*(x)=x \}.
$$
Let $r$ and $\ell$ denote the rank of $L_+$ and the length of the discriminant group $L_+^\#$, respectively. 
Artebani, Sarti, and Taki show that $L_+$ is contained in $\Pic X$~\cite[Theorem 2.1 a)]{AST11},  hence $r\leq \rho (X)$.
Cases~(1) and~(2) then follow directly from \cite[Theorem 4.2.2]{Nik83}, where it is shown that in these cases $r=10$.
The same theorem also tells us that
in case~(3), 
$\ell = 22-r-2p_a(C)$ and $k=(r-\ell)/2$.
By substituting $\ell$ in the expression of $k$, we get
$k=r-11+p_a(C),$
and hence $\rho (X) \geq r = 11-p_a(C)+k$,
completing the proof. 
\end{proof}

From this theorem, it follows that 
if one wants to find a K3 surface with a low Picard number and holomorphic involution then they should should look for non-symplectic involutions.

\section{Double covers of the plane}\label{s:DoubleCovers}

In this section, we review  the theory of K3 surfaces of degree 2 as double covers of $\IP^2$. 
Moreover, we will recall that if a complex K3 surface has Picard number $1$ and it admits a holomorphic involution, then it has to be a complex K3 surface of degree $2$.
First, we show that in order to have a double cover structure it is enough to have an ample divisor of self-intersection $2$.
This result will be key in our constructions.
We then prove~\autoref{c:rho-bounds}.

\begin{lemma}\label{l:Deg2K3}
  Let $X$ be a complex K3 surface and $H\in\Pic X$ an ample divisor with $H^2=2$.
  Then~$X$ is isomorphic to a double cover of $\IP^2$ branched along a smooth sextic curve.
\end{lemma}
\begin{proof}
    The linear system $|H|$ induces a two-to-one map $X\to \IP^2$ which ramifies above a sextic.
    As~$X$ is a K3 surface, $X$ is smooth, and hence the ramification locus is also smooth.
\end{proof}

\begin{remark}\label{r:DoubleCoverInv}
    In the case of double covers of $\IP^2$, a holomorphic involution is very easy to write down, as there is an involution induced by the double cover structure.
    More explicitly, if $X$ is such a K3 surface, then $X$ can be written as $w^2=f(x,y,z)$ in the weighted projective space 
    $\IP(1,1,1,3)$ with variables $x,y,z,w$ of weight $1,1,1,3$, respectively, and $f$ a homogeneous polynomial of degree~$6$ in three variables, defining a smooth sextic curve in $\IP^2$.
    Then the double cover $X\to \IP^2$ is given by 
    sending the point $(x:y:z:w)$ to~$(x:y:z)$ and the the double cover involution is given by
    \begin{equation}\label{eq:iota}
        \iota\colon (x:y:z:w)\mapsto (x:y:z:-w)\; .
    \end{equation}
\end{remark}

\begin{figure}[htbp]
\begin{minipage}{0.25\textwidth}
    \includegraphics[width=5cm]{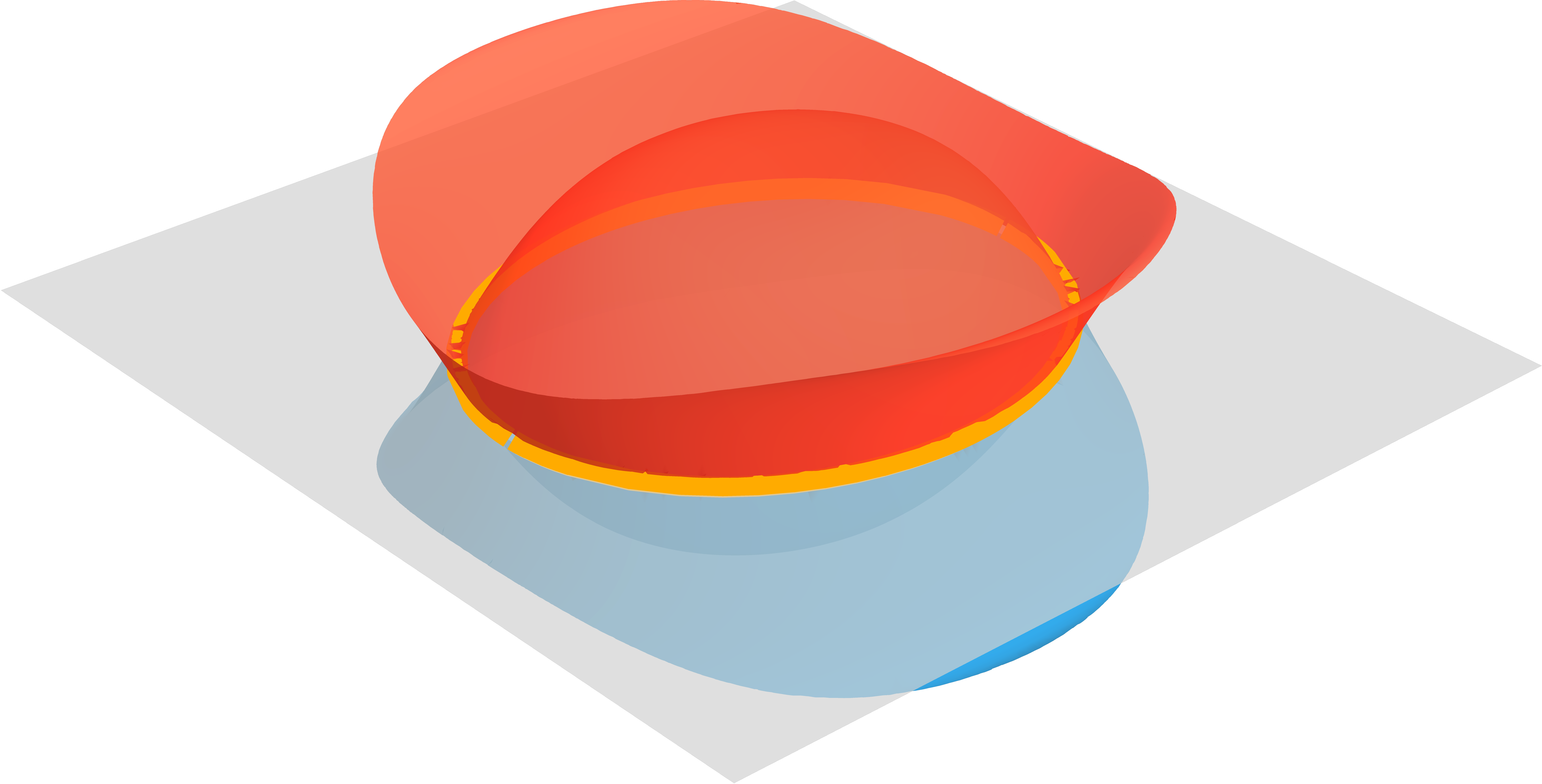}
\end{minipage}
\begin{minipage}{0.1\textwidth}
    \vspace{0.3cm}
    
    {\color{sunyellow}$Z(f)$}

    \vspace{0.3cm}
    
    {\color{darkgray}$\;\;\;\;\;\; \IP^2$}
\end{minipage}
\caption{A double cover of $\IP^2$ branched over the smooth sextic $Z(f)$ is an example of a K3 surface. It admits the holomorphic involution that swaps the sheets of the cover, coloured red and blue in the figure.}
\end{figure}

\begin{proposition}\label{p:GeneralDoubleSextic}
Let $X$ be a complex K3 surface with $\Pic X=\IZ\cdot H$.
Then the following are equivalent.
\begin{enumerate}
    \item $X$ admits a holomorphic involution.
    \item $H^2=2$.
    \item $X$ is isomorphic to a double cover of $\IP^2$ ramified above a smooth sextic.
\end{enumerate}
\end{proposition}

\begin{proof}
The equivalence $(1)\iff (2)$ is~\cite[Corollary 15.2.12]{Huy16}.

\autoref{l:Deg2K3} shows that $(2)\implies (3)$.

Now assume $(3)$, i.e., $X$ is isomorphic to a double cover $Y$ of $\IP^2$ ramified above a sextic;
denote by $\phi\colon X \to Y$ an isomorphism between $X$ and $Y$,
and let $\iota$ denote the double cover involution of $Y$ described in~\autoref{r:DoubleCoverInv}.
Then $\phi^{-1}\circ\iota\circ\phi$ is a holomorphic involution,
proving $(3)\implies (1)$.
\end{proof}

The first example of a K3 surface defined over $\IQ$ with Picard lattice isometric to $\langle2\rangle$ was given in~\cite{EJ08}.
Notice that this K3 surface, being defined over $\IQ$,  admits a holomorphic and a commuting anti-holomorphic involution, see~\autoref{p:GeneralDoubleSextic} and~\autoref{r:anti-holo}. 
Here we give another example.

\begin{example}\label{e:GenericDoubleSextic}
Consider the K3 surface 
$X_2\subset \IP(1,1,1,3)$ defined by $w^2=f$
with 
\begin{align*}
    f :=& \, 7x^6 + x^5y - x^4y^2 - 9x^3y^3 + 2x^2y^4 - 6xy^5 + 7y^6 + 3x^5z + \\
    & + 7x^4yz + 6x^3y^2z - 4x^2y^3z - 4xy^4z + 9y^5z + 2x^4z^2 + \\
    & - 4x^3yz^2 - 6x^2y^2z^2 - 7xy^3z^2 + 5x^2yz^3 + 5xy^2z^3 + \\
    & - 8y^3z^3 - 6x^2z^4 + 5xyz^4 + 8y^2z^4 + 7xz^5 - 2yz^5 + z^6.
\end{align*}
Let $X_{2,p}$ denote the base change to an algebraic closure of the reduction of $X_2$ modulo $p$.
Using the \texttt{Magma} package for K3 surfaces of degree two, based on~\cite{EJ10},
we compute the rank and discriminant of the Picard lattice of $X_{2,5}$ and $X_{2,13}$.
We obtain the rank of $\Pic X_{2,5}$ and the rank of $\Pic X_{2,13}$ both equal 2, 
and $\det \Pic X_{2,5}\not\equiv \det \Pic X_{2,13} \bmod \IQ^2$.
From this we conclude that $\rho (X_2)=1$, see~\cite[Remark 17.2.17]{Huy16}.
Notice that $X_2$ has a double cover involution and it is defined over~$\IQ$.
\end{example}

We now have all we need to prove \autoref{c:rho-bounds}.

\begin{proof}[Proof of \autoref{c:rho-bounds}]
   Let $X$ be a projective K3 surface of degree $2d$, Picard number $\rho$ and admitting a holomorphic involution.
   \autoref{p:GeneralDoubleSextic} implies that if $d=1$ then $\rho \geq 1$, if $d\geq 2$ then $\rho\geq 2$.
   To complete the proof then we only need to provide examples realizing these lower bounds.

   \autoref{e:GenericDoubleSextic} provides one with $d=1$ and Picard lattice isomorphic to $\langle 2\rangle$.

   In~\cite{Kondo1992}, Kond\={o} shows that the elliptic K3 surface 
   $$
      X_{66}\colon y^2=x^3+t(t^{11}-1)
   $$
   has Picard lattice isomorphic to $U$ and automorphism group isomorphic to $\IZ/66\IZ$, hence it admits an involution.
   We can choose generators $e,f$ of $U=[0\; 1\; 0]$ such that $e^2=f^2=0$ and $e.f=1$.
   Then the only two $-2$-classes in $U$ are $\pm (e-f)$.
   Assume $O=(e-f)$ is effective. Hence $O$ is the only effective $-2$-curve of $S$, that is, the section of the fibration.
    The positive cone of $X$ is given by divisors $xe+yf$ such that $xy>0$.
    Hence the ample cone is given by divisors $D=xe+yf$ such that $xy>0$ and $D.(e-f)>0$. 
    As $D.(e-f)=-x+y$, we obtain that $D$ is ample if and only if $y>x>0$.
    From this it follows that for every $d>1$ the divisor $e+df$ is ample and of degree $d$, showing that $X_{66}$ realizes the minimum $\rho=2$ for every $d>1$.

    We conclude by noticing that both examples are defined over $\IQ$.
\end{proof}

\begin{remark}
  Let $\cH_{p,2d}:=\{ (X,H,\alpha) \}$
denote the set of complex polarized K3 surfaces $(X,H)$ such that $H^2=2d>0$ and $X$ admits an automorphism $\alpha$ of prime order $p$.
If one defines
$$
h_{p,2d}=\min_{X\in \cH_{p,2d} } \{ \rho (X) \}\; 
$$
then \autoref{c:rho-bounds} says that 
$$
h_{2,2d}=
\begin{cases}
    1 \text{ if } d=1 \\
    2 \text{ if } d>1
\end{cases}
.
$$
\end{remark}

\section{Quartics with a node}\label{s:NodalQuartics}

In this section, we will see how a nodal quartic surface can be used to produce  an example of a K3 surface defined over $\IQ$, with   degree $2d$ and Picard number $2$ which admits a non-symplectic involution, for infinitely many values of $d$.

Let $X\subset \IP^3$ be a quartic surface with a single node, i.e., an $A_1$-singularity, and no other singularities. Then a smooth model $S$ of $X$ can be obtained by blowing up this node. This surface $S$ contains at least two prime divisors: 
the pullback $H$ of a hyperplane of~$X$,  
and the exceptional divisor $E$ over the node of $X$. The exceptional divisor is a smooth rational curve, and so~$E^2=-2$.
As the hyperplane can be assumed \emph{not} to contain the node, we have that $H^2=4$ and $H.E=0$.
This means that $\Pic S$ contains the lattice $\langle H, E\rangle \cong [4 \; 0\; {-}2]$. If~$\rho(S)=2$, then this will be the complete Picard group of~$S$. 

\begin{proposition}\label{p:NodalQuartic}
    If $\rho(S)=2$, then we have that $\Pic S=\langle H,E\rangle \cong [4 \; 0\; {-}2]$.
    Moreover, 
    \[
    \Aut S = \langle \iota \rangle \cong \IZ/2\IZ
    \]
    and for the 
    induced involution $\iota^*$ on $\Pic S$, we have $\iota^*(H)=3H-4E$ and $\iota^*(E)=2H-3E$.
\end{proposition}
\begin{proof}
    The discriminant of $[4 \; 0\; {-}2]$ is $-8=-2^3$. 
    As the discriminant of a sublattice of rank 2 can only differ by a square, 
    if $\langle H,E\rangle \subsetneqq \Pic S$, then $\det (\Pic S) =-2$. 
    It is immediate to see that there are no even lattices of rank $2$ and discriminant $-2$, 
    and so we deduce that $\Pic S=\langle H,E\rangle\cong [4 \; 0\; {-}2]$, proving the first statement.

    The automorphism group of a K3 surface with such Picard lattice has been calculated and is isomorphic to $\IZ/2\IZ$, see Bini~\cite[Theorem 1 (ii) and Example 2]{Bin05},
    where he also proves the statement about $\iota^*$.
\end{proof}

In \autoref{r:NodalInv} we give a geometric description of
how a non-trivial involution on $S$ is induced by the node on the quartic $X$.

\begin{remark}\label{r:NodalInv}
  Let $O$ denote the node on $X$.
  A non-trivial involution on $S$ is induced by a non-trivial birational involution on $X$, which can easily be constructed.
  Without loss of generality, we assume that $O=(0:0:0:1)$. 
  Hence, $X$ is defined by an equation of the form 
  \begin{equation}\label{eq:NodalQuartic}
    X\colon f_4+f_3w+f_2w^2=0,
  \end{equation}
  where $f_i$ is a homogeneous polynomial in $x,y,z$ of degree $i$. 
  Consider the projection
  \begin{equation}\label{eq:NodalDoubleCover}
  \pi\colon X\dashrightarrow \IP^2\, ,\; (x:y:z:w)\mapsto (x:y:z).
  \end{equation}
  The map $\pi$ is undefined only at $O$, 
  the pre-image of $(x:y:z)\in\IP^2$ is 
  $$
  \left\{ \left(x:y:z:\frac{-f_3\pm \sqrt{f_3^2-4f_4f_2}}{2f_2}\right)\right\},
  $$
  showing that $\pi$ is generically $2$-to-$1$ and that the branch locus $B\subset \IP^2$ is the sextic curve given by
  \begin{equation}\label{eq:NodalBranch}
      B\colon f_3^2-4f_4f_2=0\; .
  \end{equation}
  Generically~$B$ will be smooth and of genus $10$.
  The double cover involution is then the map
  $\tilde\iota\colon X\dashrightarrow X$ defined by
  $$
  (x:y:z:w)\mapsto (x:y:z:-f_3/f_2-w).
  $$
  Note that $\tilde\iota$ is undefined along $D:=X\cap \{ f_2 =0\}$.
  After blowing up the node $O$, one can extend $\tilde\iota$ to an involution $\iota\colon S\to S$ by swapping $D$ and the exceptional divisor $E$.
  This construction shows that 
  $\pi \colon S \to \IP^2$ 
  is a double cover ramified above the sextic curve $B$ and that 
  $\iota\colon S\to S$ 
  is the associated double cover involution.

  In~\cite[Example 2]{Bin05}, G. Bini shows that a K3 surface with Picard lattice isometric to $[4\; 0\; -2]$ can always be realized as in the construction above.

  Yet another way to see that $X$ doubly covers $\IP^2$ consists in considering the set $\cL$ of lines of~$\IP^3$ passing through~$O$. 
  Any line~$\ell$ in~$\cL$ intersects $X$ in $O$ with multiplicity (at least) $2$, as $O$ is a node for $X$.
  As~$X$ has degree~$4$, this means that $\ell$ intersects $X$ in two other points.
  In particular, given a point~$P\in X$, 
  one can always consider the line $\ell_P\in \cL$ passing through $P$ and $O$.
  As $\cL$ is isomorphic to $\IP^2$, the map 
  $$
  X\setminus \{O\} \to \cL\, , \;\; P\mapsto \ell_P
  $$
  gives a map that is generically $2$-to-$1$.
  After blowing up $O$, this map extends to a double cover 
  $S\to \IP^2$ as above.
\end{remark}

In the remainder, we suppose that $\rho(S)=2$. We will prove that~$S$ has a polarization of degree~$2d$ for~$d$ such that $2$ is a quadratic residue $\bmod\, d$. To show this, we will use the following lemma.
\begin{lemma}\label{l:AmpleConeS}
    Let $S$ be as above with $\Pic S=\langle H,E\rangle$. 
    The ample cone of $S$ is given by the set
    $$\cA=\left\{xH-yE : 0<y<\frac{4}{3}x\right\}.$$
\end{lemma}
\begin{proof}
By hypothesis $S$ has Picard number $2$, so we know that there are at most two $-2$-curves, see~\autoref{ss:AmpleCone}.
From \autoref{p:NodalQuartic} we know that on $S$ there are exactly two $-2$ curves, and their classes are $E$ and $2H-3E=:N$.

Set $D:=xH-yE\in \Pic S\otimes \IR$. By the discussion in~\autoref{ss:AmpleCone} it follows that $D$ is ample if and only if $D.E>0$ and $D.N>0$. In particular, if $D$ is ample, we must have $(xH-yE). E > 0$, from which it follows that $y>0$. This gives the first boundary. For the other boundary, we now deduce similarly that $0<(xH-yE).(2H-3E)=8x-6y$, which gives that $y<\frac{4}{3}x$, completing the proof.
\end{proof}

\begin{proposition}\label{p:Ample2dClass}
Let $S$ be as above with $\Pic S=\langle H,E\rangle$. Let $d>2$ be an integer such that $2$ is a quadratic residue $\bmod\, d$.
Then there is an ample class $D\in\Pic S$ such that $D^2=2d$ with $D$ primitive. 
\end{proposition}
\begin{proof}
By \autoref{l:AmpleConeS} the surface $S$ admits a primitive ample $2d$-class $D$ if and only if the Pell's equation
\begin{equation}\label{eq:2Pell}
    y^2-2x^2=-d
\end{equation}
admits a solution $(x_0,y_0)$ with $\gcd(x_0,y_0)=1$, $x_0>0$ and~$0<y_0<\frac{4}{3}x_0$.

Already in 1868, Lagrange showed
that the problem of determining the existence of solutions for the equation $x^2-Dy^2=m$ can be reduced to determining the existence of solutions for the equation $x^2-Dy^2=m'$, with $|m'|<\sqrt{D}$. 
See Shockley~\cite[Theorem 12.10.25]{Sho67} for a modern formulation and proof of Lagrange's result.

This means 
that~\eqref{eq:2Pell} can be reduced to the equation
\begin{equation}\label{eq:LagrangePell}
    y^2-2x^2=m,
\end{equation}
for some non-zero $m$ with $|m|<\sqrt{2}$.
In particular, we have that 
\eqref{eq:2Pell} has a primitive solution if and only if \eqref{eq:LagrangePell} does.
As $\sqrt{2}<2$,
it follows that $m=\pm 1$.
Both the equations $y^2-2x^2=\pm 1$ admit solutions (in fact, infinitely many),
and hence so does \eqref{eq:2Pell}.

Let $(x_1,y_1)$ be a positive primitive solution of $\eqref{eq:2Pell}$, 
that is, $x_1>0, y_1>0,$ and $\gcd(x_1,y_1)=1$.
Set $D_1:=x_1H-y_1E\in\Pic S$.
Then $D_1$ lies in the positive cone $\cP^+$ of $\Pic S$.
Moreover, as $\gcd(x_1,y_1)=1$, the class $D_1$ is primitive.
As the Weyl group of $\Pic S$ acts transitively on the set of chambers of $\cP^+$, 
there is an isometry of $\Pic S$ sending $D_1$ to a class $D_0=x_0H-y_0E$ in the closure of the ample chamber $\cA$.

Notice that, as $\gcd(x_0,y_0)=\gcd(x_1,y_1)=1$,
we have that either $x_0,y_0\neq 0$ or if one of the two is equal to $0$, the other one must be equal to $\pm 1$.
Assume $x_0=0$, then $y_0=\pm 1$ and hence $D_0^2=-2$, which contradicts the fact that $D_0$ lies in the positive cone and hence $D_0>0$;
if $y_0=0$ then $x_0=\pm 1$ and $d=2$, contradicting the hypothesis $d>2$.
So we have shown that both $x_0$ and $y_0$ are nonzero.
In order to show that $D_0$ lies inside $\cA$ and not on its border, we are left to show that $y_0$ is strictly smaller than $\frac{4}{3}x_0$.
By contradiction, assume $y_0=\frac{4}{3}x_0$. 
As $x_0,y_0\in\IN$ and $\gcd(x_0,y_0)=1$,
the only possibility is  $x_0=3,y_0=4$, from which it follows that $d=2$, again contradicting the assumption $d>2$.

We can then conclude that $D_0$ is a primitive class inside $\cA$.
That is, $D_0$ is ample and $D_0^2=D_1^2=2d$,
proving the statement.
\end{proof}

\begin{remark}
By considering the equation in \eqref{eq:2Pell} modulo $d$, 
we see that a necessary condition  to have solutions is that $2$ is a quadratic residue modulo $d$.
This means that $S$ has an ample class of degree $2d$ only if $d$ has that property.
\end{remark}

\begin{remark}
  We list the 19 values of $2< d\leq 100$ such that $2$ is a quadratic residue modulo $d$:
  $7, 14, 17, 23, 31, 34, 41, 46, 47, 49, 62, 71, 73, 79, 82, 89, 94, 97, 98$.
\end{remark}

To finish this section, we give an example of a nodal quartic that is defined over $\IQ$ for which the smooth model has Picard number $2$.

\begin{example}\label{e:NodalQuartic}
Let $X'\subseteq \IP^3$ be the nodal quartic surface defined as in~\eqref{eq:NodalQuartic} with
\begin{align*}
    f_2&:=3x^2 + 9xy + 9y^2 + 7xz + 8yz + 8z^2\, ,\\
    f_3&:=- 8x^3 + 8x^2y - 8xy^2 - 5y^3 - 9x^2z - xyz - 6y^2z + 5xz^2 - 9yz^2 - 7z^3\, ,\\
    f_4&:=- 5x^4 + 5x^3y + 4x^2y^2 + xy^3 - 4x^3z + 8x^2yz - 8xy^2z + 7y^3z \\
    &\hspace{.5cm} + 9x^2z^2 - 4xyz^2 + 2y^2z^2 - 6xz^3 + 3yz^3 - 4z^4\, .
\end{align*}
Then $X'$ is birationally equivalent to the double cover $S'$ of $\IP^2$ ramified above the sextic $B'$ defined by~\eqref{eq:NodalBranch}.
By looking at the reduction of $S'$ at $5$,
we see that $\rho (S')=2$.
Hence, $\Pic S'\cong [4 \; 0 \; {-}2]$.
As~$X'$ is defined over $\IQ$,
we conclude that $S'$ is a nodal quartic surface admitting a holomorphic as well as an anti-holomorphic involution commuting with each other.
\end{example}

\section{Smooth quartics with Picard number equal to 2}\label{s:smooth-quartics-with-picard-number-2}
In this section, we study the Picard number of quartic K3 surfaces admitting a holomorphic involution.
From~\autoref{p:GeneralDoubleSextic} 
it follows that such a K3 surface has Picard number at least $2$.

The following result gives a criterion for the existence of quartic K3 surfaces containing a curve of degree $e$ and genus $g$.
\begin{theorem}\label{t:Mori}
There exists a smooth quartic $X\subset\IP^3$ containing a smooth curve $C$ of degree $e$ and genus $g$ if and only if
\begin{enumerate}
    \item $g=e^2/8+1$ or
    \item $g<e^2/8$ and $(e,g)\neq (5,3)$.
\end{enumerate}
In this case, $X$ can be chosen such that $\Pic X$ is generated by $C$ and the hyperplane section.
\end{theorem}
\begin{proof}
The first statement is due to Mori~\cite[Theorem 1]{Mor84},
the second to Knutsen~\cite[Theorem 1]{Knu02}.
\end{proof}

We use this theorem to show the existence of smooth quartics with Picard number two admitting a holomorphic involution.

\begin{proposition}\label{p:involution}
Let $e>4$ and let $X\subset \IP^3$ be a smooth quartic containing a smooth curve $C$ of genus 2 and degree $e$ such that $\Pic X=\langle H,C\rangle$, where $H$ denotes the hyperplane section. Then~$C$ is ample and $X$ admits a holomorphic involution induced by the double cover model associated to the linear system $|C|$. 
\end{proposition}
\begin{proof}

We show that $C$ is ample. We have $C^2=2g-2=2>0$ from the adjunction formula, hence it is in the positive cone. So we are left to check that $C$ has positive intersection with all $-2$-curves.
To this end, let $D$ be a $-2$-curve and write $D=aH+bC$ with $a,b\in \IZ$. Then we have 
\begin{equation}\label{eq:-2Div}
    -2=D^2=4a^2+2eab+2b^2.
\end{equation}
As $D$ and $C$ are both effective, we know that $C.D\geq 0$.
Assume $0=C.D=ea+2b$ and hence $b=-\frac{ea}{2}$.
Plugging this equality into~\eqref{eq:-2Div}, we obtain the equation
 $a^2(-8+e^2)=4$, which has no solutions for $e>4$.
This shows that $C.D \neq 0$, so we deduce that $C.D >0$.
Therefore, $C$ must be ample.

As $C$ is ample and of genus~$2$, the linear system $|C|$ induces a model of~$X$ as double cover of~$\IP^2$ ramified above a sextic, see~\autoref{l:Deg2K3} or~\cite[Remark 2.2.4]{Huy16}.
This means that $X$ has a double cover structure.
The double cover involution on~$X$ is holomorphic, proving the statement.
\end{proof}

After proving the existence of smooth quartics with Picard number equal to $2$ and a holomorphic involution, one might ask about the nature of this involution.
Let~$H$ be the hyperplane section of~$X$. 
Then~$|H|$ induces an embedding $\phi_H\colon X\hookrightarrow\IP^3$.
The involution~$\iota$ induces an isometry on $\Pic X$ preserving the ample cone.
This means that~$\iota^*H$ is still ample and still has self-intersection $4$, but it does not need to be equivalent to $H$.
It turns out that~$\iota^*H$ is linearly equivalent to~$H$ if and only if~$\iota$ is the restriction of some involution of~$\IP^3$.

\begin{lemma}
  Let $X$ be a smooth quartic in $\IP^3$ and let $\iota\colon X \rightarrow X$ be a holomorphic involution.
  Let~$H$ be a hyperplane section of the smooth quartic $X$.
  Then $\iota ^*H$ and $H$ are linearly equivalent if and only if there exists an involution $\tilde{\iota}\colon \IP^3 \rightarrow \IP^3$ such that $\tilde{\iota}|_X =\iota$.
\end{lemma}

\begin{proof}
    If $\iota$ comes from a projective involution $\tilde\iota$, then it preserves the hyperplane section class, proving the implication from right to left.
    
    Now suppose that $\iota ^*H$ and $H$ are linearly equivalent. 
    Let $|H|$ denote the complete linear system 
    $$
    |H|:=\{D\in \Div X\; :\; D\sim H, \; D\geq0\}\; .
    $$ 
    This linear system gets the structure of a projective space by the bijection 
    $$
    \IP(H^0(X,H))\to |H|\; ; \;\;\; [f]\mapsto H+\text{div} f\; .
    $$
    The map associated to the linear system is then given by the map 
    $$
    \varphi_{|H|}\colon X\to |H|^\vee\; ; \;\;\; P\mapsto \{D\in |H| \; : \; P\in \text{Supp}(D)\}\; .
    $$ 
    The map $\iota$ induces a linear map $$H^0(X,H)\to H^0(X,\iota^*H)\; ; \;\;\; [f]\mapsto [f\circ \iota].$$ This map is an isomorphism and induces a projective linear isomorphism 
    $$
    \iota^*\colon |H|\to |\iota^* H|\; ; \;\;\; D\mapsto \iota^*(D).
    $$
    The projective dual of this map $(\iota^*)^\vee\colon |\iota^* H|^\vee\to |H|^\vee$ is given by sending a hyperplane $U\subset |H|$ to the hyperplane $(\iota^*)^{-1}(U)$. One can check that this gives the commutative diagram below.
    \[\begin{tikzcd}
        X\arrow[r]{}{\iota}\arrow[d, swap]{}{\varphi_{|\iota^* H|}} & X\arrow[d]{}{\varphi_{|H|}} \\
        {|\iota^* H|}^\vee \arrow[r]{}{(\iota^*)^\vee} & {|H|}^\vee
    \end{tikzcd}\]
    Now recall that, by assumption, $\iota ^*H$ and $H$ are linearly equivalent, which means that we can identify~$|\iota^*H|=|H|$.
    We can choose an isomorphism $\chi\colon \IP^3\xrightarrow{\sim} {|H|}^\vee$ such that the composition $\chi^{-1}\circ \varphi_{|H|}\colon X\to\IP^3$ is the inclusion $X\subset \IP^3$. Now define $\tilde{\iota}:=\chi^{-1}\circ (\iota^*)^\vee \circ \chi$. This gives the following commutative diagram.
    \[\begin{tikzcd}
        X\arrow[r]{}{\iota}\arrow[dd, swap, bend right]{}{\varphi_{|H|}}\arrow[d, hook] & X\arrow[dd, bend left]{}{\varphi_{|H|}}\arrow[d, hook] \\
        \IP^3 \arrow[r]{}{\tilde{\iota}}\arrow[d]{}{\chi} & \IP^3 \arrow[d]{}{\chi}\\
        {|H|}^\vee \arrow[r]{}{(\iota^*)^\vee} & {|H|}^\vee 
    \end{tikzcd}\]
    From the diagram we deduce that $\tilde{\iota}|_X=\iota$. Furthermore, one can deduce that $$\tilde{\iota}^2=(\chi^{-1}\circ (\iota^*)^\vee \circ \chi)^2=\chi^{-1}\circ ((\iota^*)^\vee)^2 \circ \chi =\chi^{-1}\circ ((\iota^2)^*)^\vee \circ \chi=\chi^{-1} \circ \id_{|H|^\vee}\circ \chi=\id_{\IP^3},$$
    which completes the proof.
\end{proof}

This leads to the question of whether there is a smooth quartic surface $X\subset \IP^3$ such that $\rho(X)=2$ and such that it admits a linear involution, i.e., an involution $\iota\colon \IP^3\to \IP^3$ with $\iota(X) =X$.
In the remainder of this section we show that such a quartic does not exist.

Let $\iota\colon \IP^3\to \IP^3$ be a linear involution.
Then $\iota$ can be represented as (the class of) a matrix in~$\PGL_4 (\IC)$, say $[I]$.
As $\iota$ is an involution, we may assume that $I$ is of the form
$$
I=
\begin{pmatrix}
\pm 1 & 0 & 0 & 0\\
0 & \pm 1& 0 & 0\\
0 & 0 & \pm 1& 0\\
0 & 0 & 0 & \pm 1
\end{pmatrix}.
$$
Let $(m,n)$ denote the signature of $I$.
As $\iota$ is not the identity, we have $n\geq 1$.
Also, as $I$ is only determined up to scalars, the case $(m,n)$ is equivalent to the case $(n,m)$.
We are then left with only two cases: $\sign I = (3,1)$ and $\sign I = (2,2)$.

\begin{remark}\label{r:SympSign}
    From~\cite[Lemma 15.1.4]{Huy16} it follows that if $\iota$ is symplectic, then $\sign I = (2,2)$.
\end{remark}

\begin{proposition}\label{p:InvPicNum}
Let $X\subset \IP^3$ be a quartic surface and suppose there is a linear involution $\iota\colon \IP^3\to\IP^3$ such that $\iota(X)=X$ and $\iota|_X\neq \id_X$.
Then $\rho (X) \geq 8$.
\end{proposition}
\begin{proof}
As above, let $\iota$ be represented by a matrix $I$ such that either $\sign I$ equals $(3,1)$ or $(2,2)$.

First assume that $\sign I = (3,1)$.
Then, by~\autoref{r:SympSign}, we know that $\iota$ is non-symplectic.
Choose the coordinates of $\IP^3$ such that $(1,0,0,0)$ is a basis of the eigenspace associated to the eigenvalue~$-1$ of $I$.
Then $X$ can be defined by a polynomial $f$ of the form
\begin{equation*}
f=x_0^4+x_0^2f_2+f_4,
\end{equation*}
with $f_i=f_i(x_1,x_2,x_3)$ homogeneous polynomials of degree $i$.
Then the fixed locus of~$\iota$ in~$\IP^3$ is $$\Fix \iota = \{x_0=0\} \cup \{(1:0:0:0)\}.$$
Notice that $(1:0:0:0)\notin X$.
So $\Fix\iota \cap X = \{f=x_0=0\}$ is the hyperplane section $H$ of $X$ cut by the plane $\{x_0=0\}$.
Then the irreducible components of $H$ can be
\begin{enumerate}
    \item a quartic curve, or
    \item a cubic curve and a line, or
    \item two conics, or
    \item a conic and two lines, or
    \item four lines.
\end{enumerate}
By~\autoref{t:NikulinInv}, it follows that $\rho (X)\geq 8,11,12,13,14$, respectively.

Now assume that $\sign I =(2,2)$.
If $\iota$ is symplectic, then $\rho (X)\geq 9$ by~\autoref{t:NikulinInv}.
So assume that~$\iota$ is non-symplectic.
Choose $x_0,x_1,x_2,x_3$ to be homogeneous coordinates of $\IP^3$ such that~$(1,0,0,0)$ and $(0,1,0,0)$ form a basis of the eigenspace associated to the eigenvalue $-1$ for $I$.
Then the fixed locus of $\iota$ in $\IP^3$ is  the union of two disjoint lines, namely the lines $\{x_0=x_1=0\}$ and  $\{x_2=x_3=0\}$.
Two disjoint lines intersect $X$ in
\begin{enumerate}
    \item eight points,
    \item four points and one line (i.e., one of them is contained in $X$), or
    \item two disjoint lines (i.e., they are both contained in $X$).
\end{enumerate}
By~\autoref{t:NikulinInv}, we know that only the last case is possible and it implies $\rho (X)\geq 12$.

The statement follows.
\end{proof}

From the above we deduce that 
in the case of \autoref{p:involution}, we have that $\iota^*H\neq H$. In the next lemma we show what $\iota^*H$ will be in this case. This will be useful in the example in the next section.

\begin{lemma}\label{r:X4HolomorphicInvolution}
    Let $X$ be a quartic surface with an involution $\iota\colon X\to X$ and $\Pic X=\langle H,C\rangle$ as in~\autoref{p:involution}. Then the equalities $\iota^*C=C$ and $\iota^* H=-H+eC$ hold.
\end{lemma}
\begin{proof}
    As $C$ has genus $2$ we get $C^2=2$. 
    Then $C$ induces the involution of $\Pic X$ defined by
    \[
    \iota_C\colon \alpha \to -\alpha +(\alpha . C)C\; .
    \]
    It is easy to see that this map is indeed an isometry of order $2$ that fixes $C$.
    As $\rank \Pic X=2$,
    looking at the action of $\iota_C$ on $\langle C \rangle$ and  $C^\perp$ it is easy to see that
     this is the unique isometry of order $2$ that fixes $C$.
    The double cover involution $\iota$  given by the linear system $|C|$ induces an isometry $\iota^*$ of order two on $\Pic X$
    which preserves the class of $C$, hence $\iota^*=\iota_C$.
    The statements then follow immediately, recalling that $e=C.H$ is the degree of $C$ by definition.
\end{proof}

\section{A quartic with Picard number 2 and  polarizations of degree 2 and 8}
\label{s:an-example}

In this section, we give an example of a smooth quartic surface in $\IP^3$ defined over $\IQ$ with Picard number equal to 2 admitting an involution. We also show that this surface is isomorphic to an intersection of three quadrics in $\IP^5$. Computations pertaining to this example can be found in the \texttt{Magma} file \texttt{ExampleQuartic}.

\begin{definition}
    \label{d:WimQuartic}
    Define $X_4\subset \IP^3(x,y,z,w)$ to be the smooth quartic surface defined by
    $$
    -x(x+z-w)(xw - yz)+z(x+z)(xy - z^2)+(xy + w^2)(y^2 - zw)=0\; .
    $$
\end{definition}
\begin{lemma}
The surface $X_4$ contains the following curves:
$$
\begin{array}{cc}
   C\colon
\begin{cases}
    xw - yz =0\\
    xy^2 + x^2z - z^3 - xyw + yw^2 - w^3  = 0\\
    xz^2 + z^3 + xyw + w^3 = 0
\end{cases}
\;\;\textrm{ and }\;\;  &  D\colon
\begin{cases}
    xw - yz&=0\\
    xy - z^2&=0\\
    y^2 - zw&=0
\end{cases}
\;\;.
\end{array}
$$
The curve $C$ has degree $5$ and genus $2$;
the curve $D$ is a twisted cubic, hence it has degree $3$ and genus $0$.
\end{lemma}
\begin{proof}
This follows from direct computations in \texttt{Magma}.
\end{proof}

\begin{proposition}\label{p:picX4}
The Picard lattice~$\Pic X_{4,\IC}$, is generated by the class of~$C$ and the hyperplane section and is isometric to the lattice $[4\; 5\; 2]$. Moreover, $\Pic X_{4,\IC}$ equals~$\Pic X_4$.
\end{proposition}
\begin{proof}
Let $H$ be the hyperplane section class and, by slight abuse of notation, 
let $C$ also denote the class of the curve $C$ inside $\Pic X_4$. The classes $H,C\in\Pic X_{4,\IC}$ generate a sublattice isometric to $[4\; 5\; 2]$, showing that the Picard number of $X_4$ is at least $2$.
By looking at the reduction of~$X_4$ modulo~$2$,
a verification with \texttt{Magma} gives that the Picard number of $X_4$ is as most $2$, proving that it equals 2.
As $\det [4\; 5\; 2] = -17$ is square-free and $H$ and $C$ are defined over $\IQ$, we conclude that~$\Pic X_4=\Pic X_{4,\IC}\cong [4\; 5\; 2]$.
\end{proof}

\begin{corollary}\label{c:X4involution}
    The surface $X_4$ is a realization of the surface in \autoref{p:involution} with $e=5$. In particular, the curve $C$ is ample and $X_4$ admits a holomorphic involution $\iota\colon X_4\to X_4$ such that $\iota^*(C)=C$ and $\iota^*(H)=-H+5C$ where $H$ denotes the hyperplane section class.
\end{corollary}
\begin{proof}
    This corollary follows from combining the result of \autoref{p:picX4} with \autoref{p:involution} and \autoref{r:X4HolomorphicInvolution}.
\end{proof}

\begin{remark}\label{r:WimConstruction}
    Let us explain how to produce examples such as $X_4$ above, i.e. quartic surfaces with Picard lattice isometric to $[4\; 5\; 2]$.
    \begin{enumerate}
        \item Construct curves $C$ and $D$ in $Q\subset \IP^3$, where $Q$ is the quadric defined by $xy=wz$, with the following properties.
        \begin{enumerate}
            \item On $\IP^1\times \IP^1$, define coordinates $x_0,x_1$ and $y_0,y_1$. Define the curve $D'\subset \IP^1\times \IP^1$ of bidegree $(2,1)$ by $g:=x_0^2y_1-x_1^2y_0$. This is a curve of genus 0. 
            \item Define another random smooth integral curve $C'\subset \IP^1\times\IP^1$ by $f=0$, where $f$ is some polynomial of bidegree $(2,3)$. This is a curve of genus $2$.
            \item The union $C'\cup D'$ will be defined by an equation of bidegree $(4,4)$.
            \item The image $C$ of $C'$ under the Segre embedding $\IP^1\times\IP^1 \to Q \subset \IP^3$ given by $$((x_0:x_1),(y_0:y_1))\mapsto (x_0y_0:x_0y_1:x_1y_0:x_1y_1)$$ has degree $5$ and genus $2$. Analogously, the image $D$ of $D'$ under the Segre embedding is a twisted cubic.
        \end{enumerate}
        \item The union of the two curves $C'\cup D'$ is defined by an equation of bidegree (4,4) and so the image of this union in $Q$ is given by the zero set of a quartic polynomial $h$. Hence, $h$ defines a quartic hypersurface $X\subset \IP^3$ containing both curves $C$ and $D$.
        \item The holomorphic involution on $X$ is induced by the double cover model given by the linear system $|C|$.
        \item It remains to check that the surface $X$ is smooth and that the Picard number equals 2. Generically this will indeed be the case.
    \end{enumerate}
    
    The file \texttt{ConstructionQuartic} at \href{https://github.com/danielplatt/quartic-k3-with-involution}{github.com/danielplatt/quartic-k3-with-involution} can be used to randomly generate candidates for $f$ and check if the resulting quartic K3 surface has Picard number equal to~$2$. This generates many examples of quartic K3 surfaces with commuting holomorphic and anti-holomorphic involutions and Picard number equal to~$2$. 
    
    In our case, we got the curve $C'\subset \IP^1\times\IP^1$ defined by the equation $f=0$, where
    $$f := x_0x_1y_0^3 + x_1^2y_0^3 + x_0^2y_0y_1^2 + x_1^2y_1^3.$$
    One can check that the image of $C'\cup D'$ under the Segre embedding to~$\IP^3$ is given by the equations $xw=yz$ and
    $-xy^3 + xz^3 + z^4 - x^3w - x^2zw + x^2w^2 - y^2w^2 + zw^3=0,$
    hence returning the defining equation of $X_4$, cf.~\autoref{d:WimQuartic}. 

    The anonymous referee noted that instead of the construction above, one can also construct the quartic with the curves as follows. 
    One can start with a smooth quartic defined by an equation of the form 
    $$
    q_1(xw - yz) + q_2(xy - z^2) + q_3(y^2 - zw)=0,
    $$ 
    where $q_1,q_2,q_3$ are homogeneous quadratic polynomials.
    Then the quartic by definition contains~$D$, and we can find a curve~$C$ in the other component of the hypersurface defined by $xw=yz$ (or $xy=z^2$ or $y^2=zw$). 
    Then instead of checking that $X$ is smooth, one needs to verify that the curve~$C$ is irreducible and smooth. 
    We added an alternative code \texttt{ConstructionQuarticAlt} which follows this construction. 
\end{remark}

\begin{remark}\label{r:inv5C-H}
  From the construction in \autoref{r:WimConstruction}, it follows that $2H\sim C+D$. This means that~$5C-H\sim 9H-5D$. By \autoref{c:X4involution} the hyperplane section class $H$ gets sent to the class of~$9H-5D$. This means that we can choose an isomorphism $|9H-5D|^\vee\xrightarrow{\sim}\IP^3$ in such a way that the composition $X\to |9H-5D|^\vee\xrightarrow{\sim}\IP^3$ equals the composition $X\xrightarrow{\iota} X\subset \IP^3$. 
  
  Next we take the curve $w=0$ on $X_4$ as a representative for $H$. Observe that with this choice the vector space $H^0(X,\cO_{X_4}(9H-5D))$ consists of rational functions~$\frac{f}{w^9}$, where $f$ is a polynomial of degree 9 vanishing at least five times along~$D$. We deduce that the morphism $\iota$ can be given by polynomials of degree 9. 
  Using \texttt{Magma}, we find that the involution $\iota\colon X\to X$ is given by $$(x:y:z:w)\mapsto (f_0:f_1:f_2:f_3),$$ where 
  
  {\Tiny \begin{align*}
  f_0&:=x^4y^5 - 
    5x^3y^4z^2 + 2x^2y^5w^2 + 10x^2y^3z^4 - 8x^2y^3zw^3 - x^2y^2zw^4 + x^2y^2w^5 + 6x^2yz^2w^4 - x^2z^6w + x^2z^2w^5 - x^2zw^6 + xy^8 \\
    & \hspace{0.5cm} - 7xy^6zw - xy^5w^3 + 15xy^4z^2w^2 + xy^3zw^4 - 10xy^2z^6 - 6xy^2z^3w^3 - xz^7w + xz^6w^2 - 3xz^4w^4 + 2y^5z^3w - y^4w^5 - y^3z^6 \\
    & \hspace{0.5cm} - 8y^3z^4w^2 + y^2z^4w^3 + 2y^2zw^6 + 5yz^8 + 6yz^5w^3 + z^9 - z^5w^4 - z^2w^7\; , \\
  f_1&:=x^3y^6 + 
    x^3y^5z + x^2y^7 + x^2y^6w - 4x^2y^5z^2 - 5x^2y^4z^3 - x^2y^4z^2w + 2x^2y^4zw^2 - 2x^2y^4w^3 - 3x^2y^3z^3w + 9x^2y^3z^2w^2  \\
    & \hspace{0.5cm} + 2x^2y^3zw^3 - 2x^2y^3w^4 - 8x^2y^2z^4w - 3x^2y^2z^3w^2 + 7x^2y^2z^2w^3 - 4x^2y^2zw^4 - 3x^2yz^4w^2 + 5x^2yz^3w^3 + x^2yz^2w^4 \\
    & \hspace{0.5cm} - x^2yzw^5 + x^2yw^6
    - 2x^2z^5w^2 + x^2z^3w^4 - x^2z^2w^5 - xy^7z + xy^7w - 4xy^6z^2 - xy^6zw + 2xy^6w^2 - 3xy^5z^2w + 4xy^5zw^2 \\
    & \hspace{0.5cm} + xy^5w^3 + 5xy^4z^4  - 5xy^4z^3w + 2xy^4zw^3 - xy^4w^4 + 10xy^3z^5 - xy^3z^3w^2 + xy^3z^2w^3 - xy^3zw^4 - xy^3w^5 - 5xy^2z^5w \\
    & \hspace{0.5cm} - 8xy^2z^4w^2 + 5xyz^6w + xyz^3w^4 + xyzw^6 - xz^6w^2 + xz^5w^3 - 2xz^3w^5 + y^9 - 5y^7zw - y^6zw^2 + y^6w^3 - y^5z^4 + 5y^5z^2w^2 \\
    & \hspace{0.5cm} - y^5zw^3 + y^5w^4 + y^4z^5 - y^4z^4w - 2y^4z^2w^3 + 3y^4zw^4 + 3y^3z^6 + 3y^3z^5w - y^3z^4w^2 - 9y^3z^3w^3 + y^3z^2w^4 - y^3w^6 - 2y^2z^7 \\
    & \hspace{0.5cm} + 3y^2z^6w - 4y^2z^5w^2 + 3y^2z^3w^4 - 3y^2z^2w^5 + 7yz^7w - yz^6w^2 - yz^5w^3 + 9yz^4w^4 - yz^2w^6 + yzw^7 + z^8w - z^6w^3 - z^3w^6\; , \\
  f_2&:=-x^4y^5 + 
    5x^3y^4z^2 + x^2y^7 - x^2y^5w^2 + 2x^2y^4z^2w + 2x^2y^4zw^2 - 3x^2y^4w^3 - 10x^2y^3z^4 - 4x^2y^3z^3w + 7x^2y^3z^2w^2 \\
    & \hspace{0.5cm} - 3x^2y^3zw^3 - 6x^2y^2z^4w + 7x^2y^2z^3w^2 + x^2y^2z^2w^3 + x^2y^2zw^4 + x^2y^2w^5 - 4x^2yz^5w + 2x^2yz^3w^3 - 2x^2yz^2w^4 + x^2z^4w^3 \\
    & \hspace{0.5cm} - x^2z^2w^5 - xy^7z - xy^7w - 4xy^6z^2 + 3xy^6zw + xy^6w^2 - xy^5z^3 + 11xy^5zw^2 - xy^5w^3 - xy^4z^3w + 2xy^4z^2w^2 - xy^4zw^3 \\
    & \hspace{0.5cm} - xy^4w^4 - 9xy^3z^4w - 10xy^3z^2w^3 + 10xy^2z^6 + xy^2z^3w^3 + xy^2zw^5 - xyz^5w^2 + 2xyz^3w^4 + xz^7w + y^8z - 3y^6z^2w \\
    & \hspace{0.5cm} - y^6zw^2 + y^6w^3 - 4y^5z^2w^2 + 3y^5zw^3 - y^4z^5 - y^4z^4w - 3y^4z^3w^2 + y^4z^2w^3 - y^4w^5 + 4y^3z^6 - 3y^3z^5w + 4y^3z^3w^3 \\
    & \hspace{0.5cm} - y^3z^2w^4 + 6y^2z^7 + 3y^2z^6w - y^2z^5w^2 + 6y^2z^4w^3 - y^2z^2w^5 + y^2zw^6 - yz^8 - 2yz^6w^2 - 2yz^3w^5 - z^7w^2\; ,\\
    f_3&:=-x^3y^6 - 
    x^2y^6z + x^2y^6w + 3x^2y^5z^2 + x^2y^5w^2 - 3x^2y^4z^2w - 9x^2y^3z^3w - 3x^2y^3z^2w^2 + 3x^2y^3zw^3 - x^2y^3w^4 - 7x^2y^2z^3w^2 \\
    & \hspace{0.5cm} + 14x^2y^2z^2w^3 + 2x^2y^2zw^4 - 2x^2y^2w^5 - 9x^2yz^4w^2 - 2x^2yz^3w^3 + 8x^2yz^2w^4 - 6x^2yzw^5 - x^2z^4w^3 + 2x^2z^3w^4 + x^2z^2w^5 \\
    & \hspace{0.5cm} - x^2zw^6 + x^2w^7 - 3xy^7z - 3xy^6zw + 4xy^5z^3 - xy^5zw^2 - 2xy^4z^2w^2 + 6xy^4zw^3 + xy^4w^4 - 12xy^3z^3w^2 + 8xy^3zw^4 \\
    & \hspace{0.5cm} - xy^3w^5 + 15xy^2z^5w - xy^2z^3w^3 - xy^2zw^5 - xy^2w^6 - 4xyz^5w^2 - 3xyz^4w^3 - 6xyz^2w^5 + 2xz^6w^2 + xz^3w^5 + xzw^7 \\
    & \hspace{0.5cm} - y^8w + y^6z^3 + 4y^6zw^2 - 3y^5zw^3 - 2y^4z^5 - 3y^4z^4w - 15y^4z^2w^3 - y^4zw^4 + y^4w^5 - y^3z^6 + 3y^3z^5w + y^3z^2w^4 \\
    & \hspace{0.5cm} + 6y^3zw^5 + 12y^2z^6w + 4y^2z^5w^2 - y^2z^4w^3 + 10y^2z^3w^4 + y^2z^2w^5 - y^2w^7 + 2yz^6w^2 - 6yz^5w^3 + 2yz^3w^5 - 6yz^2w^6 \\
    & \hspace{0.5cm} + 2z^7w^2 - z^6w^3 - z^5w^4 + 3z^4w^5 - z^2w^7 + zw^8 \; .
    \end{align*}}
\end{remark}

The surface $X_4$ also admits other models, given by the linear systems of ample divisors.
For example:
\begin{enumerate}
    \item as $C$ has genus $2$, the linear system $|C|$ returns a model of $X_4$ as double cover of $\IP^2$ ramified along a smooth sextic, call it $X_2'$;
    \item $|3H-C|$ returns a model of $X_4$ as intersection of three quadrics in $\IP^5$, call it $X_8$.
\end{enumerate}
Both examples are studied in the remainder of the section.

\subsection{Double cover.}
As the curve $C$ has genus $2$, the space of sections $H^0(X_4,\cO_{X_4}(C))$ has dimension $3$.
A basis of this space is explicitly computed in the provided \texttt{Magma} file, obtaining for example the basis
$$
1\, ,\; \ \frac{xy - z^2}{xw - yz}\, ,\; \frac{y^2 - zw}{xw - yz}\; .
$$
Then $C$ induces the double cover map $\phi_C\colon X_4\to \IP^2$  sending
$$X_4\ni (x:y:z:w)
\mapsto (xy - z^2: xw - yz: y^2 - zw)\in\IP^2\, .
$$
Giving $\IP^2$ the coordinates $s,t,u$, the branch locus $B$ of $\phi_C$ is given by $f=0$ with $$
f := s^6 - 2s^5t + s^4t^2 - 2s^3t^3 - 2s^2t^4 + t^6 + 4s^3t^2u - 4t^5u - 4s^4u^2 - 4st^3u^2 + 4t^4u^2 - 4t^2u^4 - 4su^5.
$$
In particular, the model $X_2'$ in $\IP(1,1,1,3)$ with coordinates $x',y',z',w'$ is given by the equation $w'^2=f(x',y',z')$. 
In this model, $\iota$ is represented by the double cover involution
$$
\iota_2\colon (x':y':z':w')\mapsto (x':y':z':-w')\; .
$$

\begin{remark}
  An isomorphism $\phi\colon X_4\to X_2'$ between the models $X_4$ and $X_2'$ can be found in the \texttt{Magma} file \texttt{ExampleQuartic}. 
  In the same file, we give an explicit formula for this isomorphism and its inverse. Moreover, the composition of $\phi^{-1}\circ \iota_2\circ \phi$ exactly gives a formula for the involution $\iota\colon X_4\to X_4$, which agrees with the formula we found in \autoref{r:inv5C-H}.
\end{remark}

\subsection{Intersection of three quadrics.}
\label{subsection:intersection-of-three-quadrics}
Set $K:=3H-C$. Note that $K^2=8$. Next, we show that it is ample by showing that it has positive intersection with the~$-2$-curves on~$X_4$. The two~$-2$-curves on~$X_4$ are given by $D\sim 2H-C$ and $\iota^*(D)\sim \iota^*(2H-C)=-2H+9C$. A direct verification shows that $K.D=1$ and $K.\iota^*(D)=103$, from which we deduce that~$K$ is ample. In particular, this means that~$K$ gives a polarization of~$X_4$ of degree~$8$.

In the provided \texttt{Magma} file, we compute a basis of $H^0(X_4,K)$ 
inducing an embedding $\phi_K\colon X_4\hookrightarrow \IP^5$,
where the image is the surface~$X_8\subseteq \IP^5$ defined by the following equations:
$$
X_8\colon
\begin{cases}
    &x_1^2 + x_0x_2 + x_2x_3 + x_1x_4 - x_2x_4 - x_3x_4 - x_1x_5 = 0\, ,\\
    &x_2^2 + x_2x_3 + x_1x_4 + x_4^2 - x_0x_5 - x_3x_5 = 0\, ,\\
    &x_1x_3 - x_0x_4 - x_3x_4 - x_5^2 = 0 \, .
\end{cases}
$$

\begin{remark}
  The map $\phi_K$ \emph{straightens} the twisted cubic $D$ to the line in $\IP^5$ defined by the equations $x_1=x_2=x_4=x_5=0$.
\end{remark}

The lattice $[4\; 5\; 2]$ does not represent $6$, i.e., there is no element with square $6$: if $6$ was represented, then there would be some integer solutions~$x,y\in \IZ$ with $4x^2+10xy+2y^2=6$; modulo 3 the only solution of this equation is the trivial solution $x=y=0$, which would mean that 3 divides both $x$ and $y$, but then $9\mid 6$, giving a contradiction. It follows that~$X$ has no polarization of degree 6 and in particular it cannot be embedded into $\IP^4$ as an intersection of a quadric and a cubic.
  
In fact, finding a K3 surface of degree $4$ with a curve of genus $2$ and a divisor with self-intersection~$6$ is less trivial than one might think.
Let $H,C$ denote the generators of the lattice $[4\; d\; 2]$, with~$H^2=4$ and $C^2=2$.
Then the smallest $d\geq 6$ for which there exists a divisor $K=aH+bC$  with $|a|,|b|\leq 1000$ and $K^2=6$ is $d=9$, for which we have: $(a,b)=(\pm 22, \mp 5), (\pm 22, \mp 193)$.
\autoref{t:Mori} implies the existence of a smooth quartic $X$ with Picard lattice isomorphic to $[4\; 9\; 2]$ generated by curves $H,C$.

\section{An example of a K3 surface of degree 6 with Picard number 2}
\label{s:another-example}
In this section,
we give an example of an intersection of a quadric and a cubic in $\IP^4$ defined over~$\IQ$ having Picard number 2 and admitting a holomorphic involution.  Computations on this example can be found in the \texttt{Magma} file \texttt{ExampleK3degree6}.

\begin{definition}\label{d:Degree6}
Let $X_6\subset \IP^4$ be the surface defined by $l_1g_1+l_2g_2+l_3g_3=0$ and $x_4^2=f$, where
\begin{align*}
    g_1&=x_0x_3-x_1x_2; & g_2&=x_0x_2-x_1^2; & g_3&=x_1x_3-x_2^2;\\
    l_1&=2x_0 + x_1 + x_2 + x_4 ; & l_2&=x_3 + x_4 ; & l_3&=x_1 + 2x_2 + 2x_3 + x_4;
\end{align*}
\[
f=2x_0^2 + x_1^2 + x_0x_2 + x_1x_2 + x_2^2 + x_1x_3  + 2x_2x_3 + x_3^2\, .
\]

\vspace{4pt}
\noindent
We define $C_6\subset \IP^4$ to be the curve defined by $x_4^2=f$ and $g_1=g_2=g_3=0$.
\end{definition}

\begin{lemma}\label{l:ExampleDegreeSix}
The surface $X_6$ is a K3 surface of degree $6$ and contains the smooth curve $C_6$.
This curve has degree $6$ and genus $2$.
\end{lemma}
\begin{proof}
It is clear from the equations that $X_6$ contains the curve $C_6$. In order to prove that $X_6$ is a K3 surface it is enough to check that it is smooth.
This and the other statements are checked by direct computations, which can be found in the \texttt{Magma} file.
\end{proof}

\begin{proposition}
The Picard lattice~$\Pic X_{6,\IC}$ equals $\Pic X_6$ and is isometric to $[6\; 6\; 2]$.
\end{proposition}
\begin{proof}
By looking at the reduction modulo $7$ of this model, we obtain that the Picard number of~$X_6$ is at most $2$.
As the hyperplane section $H$ and the curve $C_6$ generate a lattice isometric to $[6\; 6\; 2]$, we conclude that $\rho (X_6)=2$. The discriminant of this lattice is $-24=-2^3\cdot 3$. If an even lattice $[2a \; b \; 2c]$ would have discriminant $4ac-b^2=-6$, then $b^2\equiv 2 \bmod{4}$, which is impossible.
We deduce that there are no even lattices of rank $2$ and discriminant $-6$, hence $\Pic X_{6,\IC}\cong [6\; 6\; 2]$. As both $H$ and $C_6$ are defined over $\IQ$, we also conclude that $\Pic X_{6,\IC}=\Pic X_6$.
\end{proof}

\begin{lemma}\label{l:rep2mod4}
  The lattice $[6\; 6\; 2]$ does not represent any $d$ with $d\equiv 4 \bmod 6$. 
\end{lemma}
\begin{proof}
    Write $x,y$ for a basis of the lattice with $x^2=6$, $y^2=2$ and $xy=6$. Then for each element $z=ax+by$ we have that $z^2=6a^2+12ab+2b^2$. By looking modulo 6, we deduce that $z^2\equiv 2b^2 \bmod 6$ and hence there are no elements for which $z^2\equiv 4\bmod 6$.
\end{proof}

\begin{corollary}\label{c:X6involution}
    The surface $X_6$ admits a holomorphic involution.
\end{corollary}
\begin{proof}
    From \autoref{l:rep2mod4} it follows that $\Pic X_{6,\IC}$ does not represent $-2$ and hence does not contain any smooth rational curve. It follows that the ample cone equals the positive cone. In particular, the curve $C_6$ is ample. Because $C_6$ has genus $2$, it follows from~\autoref{r:DoubleCoverInv} that $X_6$ admits a holomorphic involution. 
\end{proof}

\begin{remark}
    As the lattice $[6\; 6\; 2]$ does not represent $4$, the surface $X_6$ does not have a model as a quartic in $\IP^3$.
\end{remark}

\begin{remark}\label{r:const}
    As in~\autoref{r:WimConstruction}, we explain how to find such a surface. 
    \begin{enumerate}
        \item First, we define a surface $X_6$ over $\IQ$ as in~\autoref{d:Degree6}, for arbitrary homogeneous polynomials $l_1,l_2,l_3\in \IZ[x_0,x_1,x_2,x_3,x_4]$ of degree~1 and $f\in \IZ[x_0,x_1,x_2,x_3]$ of degree~2.
        \item Next, we check if the surface $X_6$ and the curve $C_6$ defined by $g_1=g_2=g_3=0$ and $x_4^2=f$ are smooth. 
        Note that $C_6$ is contained in $X_6$. If $C_6$ is smooth, then it is by construction a double cover of a twisted cubic curve in $\IP^3$, ramified at 6 points and hence it has genus 2 and degree 6.
        \item The holomorphic involution on $X_6$ is induced by the double cover model given by the linear system $|C_6|$. 
        \item Again it remains to check that the Picard number equals 2, which is true for generic choices of $l_1, l_2, l_3$, and $f$. 
    \end{enumerate}
    
    Similarly as in the case with the quartic, the file \texttt{ConstructionK3Degree6} at \href{https://github.com/danielplatt/quartic-k3-with-involution}{github.com/danielplatt/ quartic-k3-with-involution} can be used to randomly generate more examples of such surfaces. 
    
    The above construction was brought to our attention by the anonymous referee. 
    In an earlier version, we used another construction, which used the Veronese embedding to embed a hyperelliptic curve of genus 2 in $\IP^4$. 
    Then we looked for K3 surfaces in $\IP^4$ that contained this curve. 
    This construction can still be found in the repository in the file \texttt{ConstructionK3Degree6Alt}.
\end{remark}

\section{K3 surfaces over the real numbers}\label{s:RealK3}

In~\cite{Morrison84}, using the surjectivity of the period map for K3 surfaces, Morrison shows that every even lattice of rank $1\leq r\leq 10$ and signature $(1,r-1)$ can be realized as the Picard lattice of a K3 surface over $\IC$.
In this section, we study complex K3 surfaces that can be defined over $\IR$, showing that Morrison's result still holds if we replace $\IC$ with $\IR$.
A similar construction was used in \cite[Theorem 6.1]{Rei2022}, but there the Picard group of the resulting K3 surface was not considered.
The moduli space of real K3 surfaces with a non-symplectic has been studied by Nikulin and Saito in~\cite{NS05} and~\cite{NS07}, giving a classification of the connected components. 

Suppose $X$ is a complex K3 surface that can be defined over $\IR$. Set $G:=\Gal(\IC/\IR)=\{ \id, \sigma \}$, where $\sigma$ denotes complex conjugation.
Then $\sigma$ induces an anti-holomorphic map $X\to X$  that we denote by the same symbol by abuse of notation. Furthermore, we get induced maps $\sigma^*$ on both the Picard group $\Pic X$ and on the cohomology group $H^2(X,\IZ)$, which defines a $G$-action on these groups. 

The embedding $c_1\colon \Pic X\to H^2(X,\IZ)$ is not $G$-equivariant, because the embedding is induced by the exponential sequence
which is not a sequence of $G$-sheaves. By composing with $-1$, we get the following commutative diagram, see \cite[I.4.8-11]{Sil89}.
    \begin{equation}\label{eq:cd}
        \begin{tikzcd}
        \Pic X\arrow[r]{}{c_1}\arrow[d]{}{\sigma^*} & H^2(X,\IZ) \arrow[d]{}{-\sigma^*} \\
        \Pic X \arrow[r]{}{c_1} & H^2(X,\IZ)
        \end{tikzcd}
    \end{equation}

Contrariwise, the following known result gives us a criterion to determine whether a complex K3 surface can be defined over $\IR$.

\begin{theorem}\label{t:RealK3}
A complex (algebraic) K3 surface $X$ can be defined over $\IR$ if and only if there exists an involution~$\tau$ of~$H^2(X,\IZ)$ such that:
\begin{enumerate}
    \item $\tau$ is an isometry of $H^2(X,\IZ)$;
    \item the $\IC$-extension of $\tau$ to $H^2(X,\IC)$ swaps the Hodge structure, that is,
    $\tau_\IC (H^{p,q}(X))=H^{q,p}(X)$ for $p+q=2$;
    \item $\tau_\IR (c_1(\cA))=-c_1(\cA)$;
    \item $\tau_\IR (c_1(\cP^+))=-c_1(\cP^+)$. 
\end{enumerate}
Moreover, in this case a model $X_0$ over $\IR$ with $X\cong X_0\times_\IR \IC$ can be chosen such that $\tau$ corresponds to $\sigma^*$ where $\sigma\colon X_0\times_\IR \IC\to X_0\times_\IR \IC$ is the map induced by conjugation on $\IC$.
\end{theorem}
For a proof of this result see, for example, Silhol's~\cite[Theorem VIII.1.6]{Sil89}.
The above criterion allows us to prove the following result.

\begin{proposition}\label{p:RealRepresentation2}
    Let $N$ be a primitive sublattice of $U^{\oplus 2} \oplus E_8(-1)^{\oplus 2}$ of rank $r$ such that $N$ has signature~$(1,r-1)$. Then there exists a K3 surface $X$ over $\IR$ with $X(\IR)\neq \emptyset$ and $\Pic X=\Pic X_\IC \cong N$ as lattices.
\end{proposition}

\begin{proof}
    Set $\Lambda := U^{\oplus 2} \oplus E_8(-1)^{\oplus 2}$ and note $\Lambda_{K3} = U \oplus \Lambda$.
    Let $N^\perp$ be the orthogonal complement of~$N$ inside~$\Lambda$. The lattice $\Lambda$ has signature $(2,18)$ and, by the assumption, $N$ has signature $(1,r-1)$, so the lattice $N^\perp$ has signature $(1,19-r)$. This means that we can pick an element~$w_-\in N^\perp\otimes\IR$ such that $w_-^2>0$. Moreover, if we choose~$w_-$ generic enough, then $w_-. v\neq 0$ for all~$v\in N^\perp\setminus \{0\}$. Observe that~$U$ has signature $(1,1)$. Then, as above, we can pick an element~$w_+\in U\otimes \IR$ such that~$w_+^2=w_-^2$ and~$w_+. v\neq 0$ for all $v\in U \setminus \{0\}$.
    Notice that $w_+.w_-=0$.
    
    Now define $w:=w_+ + iw_-\in (U \otimes \IR)\oplus i(\Lambda \otimes \IR)\subset\Lambda_{K3}\otimes \IC$. Then $$w^2=w_+^2-w_-^2+2iw_+.w_-=0 \ \text{ and } w.\overline{w}=w_+^2+w_-^2>0.$$ This means that $w$ satisfies the Riemann conditions \eqref{eq:Rc}, and hence $w\in \Omega$ is a period. By the surjectivity of the period map, see \autoref{ss:moduli}, there exists a complex (analytic) K3 surface~$Y$ with a marking~$\alpha$ such that~$\alpha(\omega_Y)=w$.
    Moreover, the Hodge structure of this K3 surface $Y$ under this marking is given by $H^2(Y,\IC)=\bigoplus_{p+q=2} H^{p,q}(Y)$, where $H^{2,0}(Y)= \langle w\rangle $, $H^{0,2}(Y)=\langle \overline{w}\rangle $ and~$H^{1,1}(Y)=\langle w_+,w_-\rangle ^\perp$. 
    
    By the Lefschetz theorem on~$(1,1)$-classes, \cite[p.163]{Gri94}, we have that~$\Im (c_1)= H^2(Y,\IZ)\cap H^{1,1}(Y)$. Under the map~$\alpha$, we can identify $H^2(Y,\IZ)$ with $\Lambda_{K3}$. With this identification we have that 
    $$\Im (c_1)=H^2(Y,\IZ)\cap H^{1,1}(Y)= \Lambda_{K3}\cap \langle w_+,w_-\rangle ^\perp=N.$$ 
    By the injectivity of $c_1$, we deduce that $\Pic Y$ and  $N$ are isomorphic as lattices. Because there is a class~$d\in \Pic Y$ with~$d^2>0$, it follows from~\cite[Theorem~IV.6.2]{Bar2015} that~$Y$ is projective and hence algebraic.
    
    Next, define the involution~$\tau$ on~$H^2(Y,\IZ)=\Lambda_{K3}=U\oplus \Lambda$ by~$\id \oplus -\id$, i.e., it acts as the identity on~$U$ and as multiplication by~$-1$ on~$\Lambda$. We will verify that~$\tau$ satisfies all the conditions in~\autoref{t:RealK3}. 
    First of all, $\tau$ clearly is an isometry. It sends $w$ to $\overline{w}$, hence $\tau_\IC (H^{0,2}(Y))=H^{2,0}(Y)$, and vice versa.
    We also have that $$\tau_\IC(H^{1,1}(Y))=\langle \tau(w_+),\tau(w_-)\rangle^\perp=\langle w_+,-w_-\rangle^\perp=H^{1,1}(Y).$$ Hence, the involution~$\tau$ satisfies condition (2). Because both~$c_1(\cA)$ and~$c_1(\cP^+)$ are subsets of the lattice $N$ and $\tau$ acts as~$-1$ on~$N=\Im (c_1)$, conditions~(3) and (4) are satisfied. From~\autoref{t:RealK3} it follows that there is a K3 surface~$X$ over $\IR$ such that~$X\times_\IR \IC\cong Y$ and such that the action of $G$ on $H^2(X,\IC)$ is given by $\tau$.
    
    The next step is to show that $X(\IR)\neq \emptyset$. 
    Recall that $G:=\Gal(\IC/\IR)$. 
    By \cite[Proposition I.2.1]{Sil89}, 
    we have for the Euler characteristic $\chi(X(\IR))$ that
    $$
    \chi(X(\IR))=\sum_{i\equiv 0 \bmod 2} \ 2\cdot \dim H^i(X,\IC)^G -\dim H^i(X,\IC),
    $$ 
    where $\dim H^2(X,\IC)^G$ denotes the $G$-invariant subspace of $\dim H^2(X,\IC)$.
    In our case, we have that $\dim H^0(X,\IC)^G=1$, because $X$ is connected and the unique connected component is therefore preserved by $G$.  Furthermore, $H^2(X,\IC)^G=((U\oplus \Lambda)\otimes\IC) ^G=U\otimes \IC$ so it follows that $\dim (H^2(X,\IC))^G=2$.
    Lastly, we have that $\dim H^4(X,\IC)^G=1$, because complex conjugation is orientation preserving in complex even dimension.
    We thus get 
    $$\chi(X(\IR))=2\cdot 1-1+2\cdot 2 -22+2\cdot 1-1=-16.$$
    Because the Euler characteristic is non-zero, we deduce that $X(\IR)\neq \emptyset$.
    
    Now to conclude the proof that $\Pic X=\Pic X_\IC \cong N$, we first note that we can identify $\Pic X_\IC$ with the lattice $\Pic Y\cong N$. Hence, we are left to prove that $\Pic X=\Pic X_\IC$.
    We denote by $\sigma\colon X\times_\IR \IC\to X\times_\IR \IC$ the map induced by complex conjugation. Then~$\sigma^*$ corresponds to the map~$\tau$ we constructed above and with this identification we get the following commutative diagram, see~\eqref{eq:cd}.
    $$\begin{tikzcd}
        \Pic X_\IC\arrow[r]{}{c_1}\arrow[d]{}{\sigma} & H^2(X,\IZ) \arrow[d]{}{-\tau} \\
        \Pic X_\IC \arrow[r]{}{c_1} & H^2(X,\IZ)
    \end{tikzcd}$$
    Recall that $\Im (c_1)=N$ and observe that the restriction of $-\tau$ to $\Im (c_1)$ is the identity. We deduce the equality $\sigma=\id$. By the fact that $X(\IR)\neq \emptyset$, it follows that the Picard functor is representable and therefore that $\Pic X=(\Pic X_\IC)^G$ (see, e.g., \cite[Proposition~I.4.5]{Sil89}). Combining these results we get $\Pic X=\Pic X_\IC$, which completes the proof.
\end{proof}

We get the following Corollary.

\begin{corollary}\label{t:RealRepresentation}
    For any even lattice $N$ of rank $r$ with $1\leq r \leq 10$ with signature $(1,r-1)$, there exists a K3 surface $X$ over $\IR$ with $X(\IR)\neq \emptyset$ such that both $\Pic X$ and $\Pic X_\IC$ are isomorphic to~$N$ as lattices.
    Moreover, the primitive embedding $N\hookrightarrow H^2(X,\IZ)$ is unique up to isometries of~$H^2(X,\IZ)$.
\end{corollary}

\begin{proof}
    By Nikulin's~\cite[Corollary~1.13.4]{Nik80}, the lattice $\Lambda:= U^{\oplus 2} \oplus E_8(-1)^{\oplus 2}$ is the unique unimodular even lattice of signature~$(2,18)$. Note that we have~$\frac{2+18}{2}\geq r$ and so by~\cite[Theorem~1.12.4]{Nik80} it follows that there exists a primitive embedding~$N\hookrightarrow \Lambda$. By~\autoref{p:RealRepresentation2} the first statement of \autoref{t:RealRepresentation} follows.
    
    To obtain the second statement, note that composition $N\hookrightarrow \Lambda\hookrightarrow \Lambda_{K3}$ gives a primitive embedding of $N$ into $\Lambda_{K3}$. Recall that $N$ has signature $(n_+,n_-)=(1,r-1)$ with $r\leq 10$ and $\Lambda_{K3}$ has signature $(l_+,l_-)=(3,19)$. Hence, we have $l_+>n_+$ and $l_->n_-$, and $2r+2\leq 22$. We conclude that this embedding satisfies the conditions of \cite[Theorem~1.14.4]{Nik80} and it follows that this embedding is unique up to isometry.
\end{proof}

Using the above result one can find many more examples of polarized K3 surfaces with Picard number 2. As an application, we prove that for any $d>3$ there exists a polarized K3 surface of degree $2d$ defined over $\IR$ with Picard lattice $[2\; d{+}1\; 2d]$, such that it is generated by ample classes.

\begin{lemma}\label{l:lat-2}
    For every integer $d>3$, the lattice $[2 \; d{+}1 \; 2d]$ does not represent $-2$. 
\end{lemma}
\begin{proof}
    Write $L:=[2 \; d{+}1 \; 2d]$ and pick generators $x,y\in L$ such that $x^2=2$, $y^2=2d$ and $x.y=d+1$. Then to derive a contradiction, suppose that $c:=ax+by\in L$ satisfies $c^2=-2$. Note that $$c^2=(ax+by)^2=2a^2+2(d+1)ab+2db^2$$ and so we deduce that $$-1=a^2+(d+1)ab+db^2=(a+b)(a+db).$$
    This means that $a+b=\pm 1$ and $a+db=\mp 1$. It follows that $$(d-1)b=(a+db)-(a+b)=\pm 2.$$ So in both cases we get that $d-1$ divides $2$ which contradicts the assumption $d>3$. We deduce that $L$ does not represent $-2$.
\end{proof}

\begin{proposition}\label{p:RealK3RealPic}
    For $d>3$, there exists a K3 surface $X$ over $\IR$ such that 
    $$
    \Pic X=\Pic X_\IC= \langle C,H \rangle 
    \cong [2\; d+1\; 2d]  $$ 
    with $C$ and $H$ both ample and satisfying $C^2=2$, $H^2=2d$ and $C.H=d+1$.
\end{proposition}
\begin{proof}
    By \autoref{t:RealRepresentation}, there exists a K3 surface $X$ over $\IR$ with $$\Pic X=\Pic X_\IC \cong [2 \; d{+}1 \; 2d].$$ 
    We can choose generators $C,H$ satisfying $C^2=2$, $H^2=2d$ and $C.H=d+1$. 
    By possibly replacing $C,H$ with $\pm C, \pm H$, we can assume that both $C$ and $H$ are in the positive cone. 
    By combining \autoref{r:no-2} with \autoref{l:lat-2} it follows that the ample cone $\cA$ equals the positive cone $\cP^+$, and hence $C$ and $H$ are both ample.
\end{proof}

\bibliographystyle{amsplain}
\bibliography{references}

\end{document}